\documentclass[10pt]{article}

\usepackage{graphicx}
\usepackage{amsmath}
\usepackage{amsfonts}
\usepackage{amssymb}
\usepackage{amsthm}
\usepackage{eucal}
\usepackage{nccmath}
\usepackage{array}
\usepackage{bbm}
\usepackage[margin=1.25in]{geometry}

 \newtheorem{theorem}{Theorem}
\newtheorem{lemma}{Lemma}
\newtheorem{conjecture}{Conjecture}
\newtheorem{definition}[lemma]{Definition}

\newtheorem{proposition}[lemma]{Proposition}
 
\newcommand{\one}[0]{\mathbbm{1}}

\newcommand{\wt}[1]{\widetilde{#1}}
\newcommand{\ds}[0]{\displaystyle}
\newcommand{\eqd}[0]{\,\substack{ d \\ \ds = \\ \,}\,}
 
 \newcommand{\tr}[0]{\mathrm{tr}\,}
\newcommand{\E}[0]{\mathbb{E}}
\renewcommand{\t}[0]{\mathfrak{t}}
\newcommand{\ol}[1]{\overline{#1}}
\newcommand{\ul}[1]{\underline{#1}}
\newcommand{\ola}[1]{\overleftarrow{#1}}
\newcommand{\ck}[1]{\left\lceil #1 \right\rceil_n}
\newcommand{\fk}[1]{\left\lfloor #1 \right\rfloor_n}

\newcommand{\RN}[1]{\textup{\uppercase\expandafter{\romannumeral#1}}}

\title{Feynman-Kac formula for the stochastic Bessel operator}
\author{Patrick Waters\footnote{Temple University}}

\begin{document}

\maketitle
\abstract{We introduce a stochastic process and functional that should describe the semigroup generated by the stochastic Bessel operator.
Recently Gorin and Shkolnikov \cite{GS16} showed that the largest eigenvalues for certain random matrix ensembles with soft edge behavior can be understood by analyzing large powers of tridiagonal matrices, which converge to operators in the stochastic Airy semigroup. 
In this article we make some progress towards realizing Gorin and Shkolnikov's program at the random matrix hard edge.  
We analyze large powers of a suitable tridiagonal matrix model (a slight modification of the $\beta$-Laguerre ensemble).
For finite $n$ we represent the matrix powers using Feynman-Kac type formulas, which identifies a sequence of stochastic processes $X^n$ and functionals $\Phi_n$.
We show that $\Phi_n(X^n)$ converges in probability to the limiting functional $\Phi(X)$ for our proposed stochastic Bessel semigroup.
We also discuss how the semigroup method may be used to understand transitions from a hard edge to a soft edge in the $\beta$-Laguerre models. 

%
}

\section{Introduction}

%

\subsection{$\beta$-ensembles and their differential operator limits}

Perhaps the best known random matrix models are the classical ``invariant ensembles". 
In these models the random matrices have real, complex or quaternion entries and thus it makes sense to talk about the number of degrees of freedom $\beta=1,2,4$ for the matrix entries.
However the joint densities of eigenvalues depend on $\beta$ in such a way that one is led to consider values of $\beta$ other than $1,2,4$.
For example the $\beta$-Laguerre eigenvalue density is
\begin{align} \label{Lag1}
dP(\lambda_1,\ldots , \lambda_n) = \frac{1}{Z_{\beta,a}} |\Delta( \lambda)|^{\beta} \prod_{k=1}^n \lambda_{k}^{\frac{\beta}{2}(a+1)-1} e^{-\frac{\beta}{2} \lambda_k} \, d\lambda_k,
\end{align}
where $\Delta(\lambda)$ is the determinant of the Vandermonde matrix $(\lambda_{i}^{j-1})_{i,j=1}^n$.  
There are also $\beta$-Hermite and $\beta$-Jacobi ensembles.
In these models the eigenvalues can be thought of as an interacting particle system called a ``log gas" \cite{ForresterBook}, and in this interpretation $\beta$ is the reciprocal of temperature.

In the celebrated paper \cite{DE02}, Dumitriu and Edelman used Householder conjugations to put the classical invariant ensembles in tridiagonal form, and showed that these tridiagonal random matrices can be constructed from independent samples from Gaussian or Chi distributions.
For example the tridiagonal $\beta$-Laguerre ensembles are $n\times n$ matrices given by
\begin{align} \label{eq1}
L=& n^{-1} BB^t ,\qquad  
B_{kk}\eqd \chi [(k+a)\beta]/\sqrt{\beta}, \qquad
 B_{k+1,k}\eqd-\chi [k\beta]/\sqrt{\beta} .
\end{align}
All other entries of $B$ are $0$.
Here $\chi[k]$ indicates the Chi distribution with parameter $k$.
The tridiagonal models naturally make sense for all $\beta>0$ and their eigenvalues agree with the general $\beta$ extensions of the classical eigenvalue densities, for example (\ref{Lag1}).

As is usual in random matrix theory, we are interested in limits taken as the matrix size $n$ approaches infinity.
Edelman and Sutton \cite{ES07} showed that in this limit matrices from the tridiagonal $\beta$-ensembles act approximately like differential operators on appropriate $L^2$ spaces, and introduced the stochastic Airy and Bessel operators.
This observation sparked a series of papers
\cite{RRV,RR09,RR17,KRV13,VV16,RW16}
exploring differential operator limits of the $\beta$-ensembles.

In this paper we are interested in limits associated with the spectral edges of the $\beta$-ensembles.
This excludes the sine-$\beta$ operator \cite{VV16} which has interesting connections to the Riemann hypothesis.
In the context of the Laguerre $\beta$-ensembles, the spectral edges can be understood as follows.
As $n\rightarrow \infty$ the mean density of eigenvalues for $L$ converges weakly in distribution to a point mass on the following Marchenko-Pastur law:
\begin{align}
d\mu(\lambda) =&\frac{1}{2\pi} \sqrt{(4-\lambda)/\lambda}\,\one_{[0,4]}(\lambda)\, d\lambda.
\end{align}
The matrices $L$ in display (\ref{eq1}) are positive definite, so the Laguerre ensembles are said to have a ``hard edge" at $\lambda=0$.
In contrast, for finite values of $n$ the largest eigenvalue of $L$ may be greater than $4$, so the ensemble is said to have a soft edge at $\lambda=4$.
Ramirez, Rider and Virag \cite{RRV} showed that the largest eigenvalues converge in finite dimensional distributions to those of the stochastic Airy operator (SAO) which we denote $\mathcal{A}$.
For some constant $c$ they showed that
\begin{align}
cn^{2/3}\left(4I -L\right) \rightarrow  \mathcal{A}=- \partial_{x}^2 +x +\frac{2}{\sqrt{\beta}}W'(x)  . \label{SAO1}
\end{align}
Here $W'(x)$ is Gaussian white noise. 
One of the achievements of \cite{RRV} was to rigorously define the eigenvalues of $\mathcal{A}$.
In the same paper they also characterized a $\beta>0$ generalization of the Airy point process in terms of a diffusion related to the eigenfunction equation for the SAO and used this to calculate tail asymptotics of the Tracy Widom-$\beta$ distributions.
Later Krishnapur, Rider and Virag \cite{KRV13} showed that this limit is universal in the sense that in describes the largest eigenvalues for a family of models indexed by convex polynomial potential functions.

The limiting operator for the hard edge is the stochastic Bessel operator $\mathcal{B}$.
Ramirez and Rider \cite{RR09} showed that the smallest eigenvalues of $n^2 L$ converge\footnote{
The $n^2$ scaling is suggested by the following formal calculation: if $\lambda_k$ is the $k^{th}$ smallest eigenvalue then one would expect $k/n \approx \int_{0}^{\lambda_k} d\mu \approx C \sqrt{\lambda_k}$.  A similar calculation explains the $n^{2/3}$ scaling in (\ref{SAO1}).  
}
 in finite dimensional distributions to those of the operator
\begin{align}
\mathcal{B}=&-x\partial_{x}^2 -\partial_x + \frac{a^2}{ 4x} -2\sqrt{\frac{x}{\beta}}W'(x)\partial_x + \frac{a}{\sqrt{\beta x}}W'(x)
.\label{SBO1}
\end{align}
Later it was shown \cite{RW16} that the SBO limit is observed universally at the hard edge for a family of models indexed by polynomial potential functions.


\subsection{Semigroup formula for the hard edge}

We now briefly discuss the results of Gorin and Shkolnikov \cite{GS16}.
They defined a family of tridiagonal random matrix models, which includes the $\beta$-Hermite ensembles, and showed that as $n\rightarrow \infty$ the semigroups generated by the tridiagonal matrices converge in distribution to a stochastic Airy semigroup.
As a consequence they concluded the following:
\begin{theorem}[Gorin-Shkolnikov \cite{GS16}] \label{GSTheorem} For almost all $W$ the following holds. If $f\in L^2[0,\infty)$ then
\begin{align}
\exp(-t \mathcal{A}/2) f(x) =& \mathbb{E}_x\left[ \one_E \exp\left( -\frac{1}{2}\int_{0}^t B_s\, ds + \frac{1}{\sqrt{\beta}}\int_{0}^{\infty}L_B(x,t)\,dW(x)\right)f(B_t)\right], \label{GS1}
\end{align}
where $B$ is a standard Brownian motion.  The subscript $\mathbb{E}_x$ indicates the starting position of $B$, and $E$ is the event that $\inf_{0\leq u\leq t}B_u\geq 0$.
\end{theorem}
In a typical statement of the Feynman-Kac formula the hypotheses would include a smoothness condition for the potential function, for example a uniform Lipschitz condition.
With the white noise term $\beta^{-1/2} W'(x)$, the operator $\mathcal{A}$ does not meet the hypotheses of any such theorem.
Furthermore the authors of \cite{GS16} showed that the traces of both sides in equation (\ref{GS1}) agree.
This extends to general $\beta$ the following formula of Okounkov \cite{OkRMRP}:
\begin{align}
\mathbb{E}\left[\exp \left\{ tn^{1/6} \tr ( H -2\sqrt{n}I)\right\}\right] \rightarrow \frac{\exp(t^3/12)}{2\sqrt{\pi}t^{3/2}},
\end{align}
where $H$ is an $n\times n$ Gaussian unitary ensemble matrix.
By comparing the SAO semigroup functional to Okounkov's formula, Gorin and Shkolnikov were able to deduce a new identity for the local time of a Brownian excursion.
In a subsequent paper Lamarre and Shkolnikov \cite{LS17} extended this technique to study the largest eigenvalues of spiked Laguerre ensembles.

In this article we make some progress towards realizing Gorin and Shkolnikov's program at the random matrix hard edge.  
We will analyze large powers of a suitable tridiagonal matrix model, a slight modification of the Laguerre $\beta$-ensemble.
For finite $n$ we represent the matrix powers using Feynman-Kac type formulas, which identifies a sequence of stochastic processes $X^n$ and functionals $\Phi_n$.
We show that $\Phi_n(X^n)$ converges in probability to the $\Phi(X)$ where
\begin{align}
dX=&\sqrt{2X}\,dB +\left(1+2\sqrt{X/\beta} W'(X)\right)\, dt,  \label{SDE1} \\
\Phi(X)=&-\frac{a^2}{4}\int_{0}^t \frac{du}{X_u} - \frac{a}{\sqrt{\beta}} \int_{0}^1   \frac{L_X(x,t)}{\sqrt{x}} \circ dW(x). \label{PhiDef1111}
\end{align}
Typical existence and uniqueness theorems for SDE's require the drift coefficient to be uniformly Lipschitz, so at first it is not clear that (\ref{SDE1}) defines a stochastic process.  
In section \ref{SectionInterpretations} we use the Ito-McKean speed and scale construction to define a process $X$ that formally satisfies (\ref{SDE1}).
It was recently explained in \cite{HLM16} how to rigorously interpret the speed and scale construction as a strong solution to (\ref{SDE1}).
Even after we explain the precise definition of the process $X$, it will not be obvious that the Stratonovich integral on line (\ref{PhiDef1111}) is well defined, but (as we will show in section \ref{SectionInterpretations}) this point can also reconciled using the methods of \cite{HLM16}.

Formally applying the Feynman-Kac formula to the SBO leads to the following conjecture:
\begin{conjecture}
The stochastic Bessel operator satisfies the following Feynman-Kac formula.  If $f\in L^2[0,1]$ then
\begin{align}
\exp(-t\mathcal{B}) f(x) \label{SBO_semigroup}
=& \mathbb{E}_x \left[ \one_A \exp \Phi(X)  f(X_t)
\right], 
\end{align}
where $\Phi$ and $X$ are as in displays (\ref{SDE1},\ref{PhiDef1111}), and $A$ is the event $\sup_{0\leq u\leq t}X_u\leq 1$.

Furthermore if $f\in L^2[0,1]$ is fixed and $f_n$ is the orthogonal projection of $f$ onto the span of $\{ \one_{[(k-1)/n,k/n)}:\; k=1,\ldots ,n\}$, then $(I-L)^{tn^2} f_n $ should converge in distribution to the right hand side of (\ref{SBO_semigroup}), with $L$ being a matrix from the $n\times n$ tridiagonal $\beta$-Laguerre ensemble.
\end{conjecture}

In this paper we make some progress towards the above conjecture.
We have been somewhat conservative in only stating the above conjecture for the $\beta$-Laguerre ensembles.
It seems very plausible that it would also hold for the models considered in \cite{RW16}, and for perturbations of the $\beta$-Laguerre ensembles where the matrix entries of $B$ are independent and satisfy some moment conditions as in \cite{GS16}.

The model we consider will actually be a slight perturbation of the $\beta$-Laguerre ensemble.
As we discuss in section \ref{DefSection}, we have chosen the model deliberately to simplify our analysis.
We let $B_n$ be a bidiagonal matrix with diagonal entries $\mathbf{x}_1,\ldots ,\mathbf{x_n}$ and subdiagonal entries $-\mathbf{y}_1,\ldots ,-\mathbf{y}_{n-1}$.
We will define the entries $\mathbf{x}_k$ and $\mathbf{y}_k$ in section \ref{matrixModelSection}, but for now we just comment that they are approximately independent and approximately distributed like the $\chi$ random variables in display (\ref{eq1}).
We then define
\begin{align}
A=\frac{1}{n} P B B^t P^{-1}, 
\end{align}
where $P$ is an $n\times n$ diagonal matrix with $P_{k,k}=(\mathbf{y}_k/\mathbf{x}_k)\sqrt{1+a/j}$.
This conjugation is necessary because formal calculations suggest that the differential operator limit for $n^{-1}BB^t$ would have a white noise squared term, and therefore not even make sense as a mapping of functions to distributions.

In the interest of stating our main theorem in this section, we now introduce the stochastic processes $X^n$ and functionals $\Phi_n$ that appear in our discrete Feynman-Kac formula.
However, to be brief we will point the reader to later sections for a couple of the details.
The processes $X^n$ are left continuous, piecewise constant and have the form $X^{n}(i \Delta t)=X(\t_i )$ where $\Delta t=1/(4tn^2)$ and $\t_i$ are some stopping times defined in section \ref{DiscreteProcessSection}.
As we show in section \ref{ApproxJumpTimesSection}, if $i$ and $n$ are large then $\t_i$ and $i\Delta t$ are approximately equal.
The discrete functionals are
\begin{align} \label{phiFormula00021}
\Phi_n(t)=& 
\frac{-1}{4n^2} \sum_{i=1}^{\lfloor 4tn^2 \rfloor} 
  \frac{a^2}{4 x(i)}+\frac{a\sqrt{n}G_{n x(i)}}{\sqrt{\beta x(i)}} +\frac{a  G^{(2)}_{nx(i)} }{\beta x(i)} ,\qquad \text{where }x(i)=X^{n}(i\Delta t),
\end{align}
and the quantities $G_k$ and $G_{k}^{(2)}$ are random variables constructed from the Wiener process $W$ in display (\ref{GDEF}) of section \ref{phiNSection}.

We will show in section \ref{discFKSection2} that the random matrices $A$ satisfy the following discrete Feynman-Kac formula.
If $v \in \mathbb{R}^n$ then
\begin{align}\label{DFK2}
((I-A/4)^{4tn^2} v)_k =\mathbb{E}_{k/n}\left[ \one_{E} \exp  \left(\Phi_n +\mathcal{E}_n\right)v_{nX^n(t)} \right],
\end{align}
where $E$ is the event $\sup_{u\leq t}X^{n}_u\leq 1$ and $\mathcal{E}_n$ is an error term which, on the event $\inf_{u\leq t}X^{n}_u \geq 1/\log n$, satisfies the estimate $\mathcal{E}_n =O(\sqrt{\log (n)/n})$.
The subscript $\mathbb{E}_{k/n}$ indicates the starting position for $X^n$.
In section \ref{HardEdgeSection} we show that the smallest eigenvalues of $n^2A$ converge to those of $\mathcal{B}$.
Assuming this, it is easy to see that the eigenvalues of the left hand side $(I-A/4)^{4tn^2}$ in (\ref{DFK2}) should converge to those of $\exp(-t\mathcal{B})$.
The left hand side of \ref{DFK2} is the same kind of matrix power considered in \cite{GS16}.

At this point it seems appropriate to comment that throughout the paper, the quality of our approximations will degrade whenever $X(t)$ approaches $0$.
That is one of the main technical difficulties that must be overcome in order to prove conjecture 1.
Our choice of the cutoff $X>1/\log n$ is somewhat arbitrary; we could for example use a cutoff of $X>n^{-p}$ for some $p>0$, but this would change the error estimates in all our propositions.

Our main result is the following.

\begin{theorem} \label{mainThm}
On the event $\inf_{u\leq 1.01t} X(u)\geq 1/\log n$, we have
\begin{align}
\left|\Phi -\Phi_n\right| \leq Ct n^{-0.249}.
\end{align}
This holds for almost all $B,W$, and $C$ depends only on $B,W,\beta$.
\end{theorem}
Below in lemma \ref{pathLowerBoundLemma} we show that given any fixed starting position, the probability of $X$ hitting $1/\log n$ before $1$ converges to zero.
Also by lemma \ref{TimeApproxProp} we have that if $X^n(t)$ converges to $X(t)$.
It follows that if $f$ is continuous on $[0,1]$ and we let
\begin{align}
v_k =n\int_{(k-1)/n}^{k/n} f(x)\, dx,
\end{align}
then the quantity inside the expected value of our discrete Feynman-Kac formula (\ref{DFK2}) converges in probability to the integrand of our proposed limiting Feynman-Kac formula (\ref{SBO_semigroup}).

%

\subsection{Semigroup in terms of Brox's diffusion}
Before moving on with the main thrust of the paper, we would like to comment briefly on one more point. 
In just this section we will relax the mathematical rigor of our exposition and present some formal calculations.
A connection to Brox's diffusion at the random matrix hard edge has been observed before in \cite{RR09}.
The SDE for our process $X$ is not quite Brox's diffusion, but formally can be transformed to it as follows.
Let 
\begin{align}
Y_t =& -\frac{1}{\beta}\log   X_{\theta(t)}  ,\qquad \text{where}\qquad\theta(t)= \beta^2 \int_{0}^t \exp\left (  -\beta Y_u \right)\, du.
\end{align}
Formally applying Ito's lemma gives the SDE
\begin{align}
dY=dB+2 W'(Y)\, dt,\label{BroxSDE}
\end{align}
but to see this, one must be careful to use the special chain rule for Brownian motion\footnote{
This chain rule can be seen from $\text{Var}\int f(x)\,dW(x) 
=\int f(x)^2 dx 
= \text{Var}\int f(h(y))\sqrt{dx/dy}\,d\wt{W}(y)$.
}:
if $x=h(y)$ then $W'(x)=|h'(y)|^{-1/2}\wt{W}'(y)$ for a new Brownian motion $\wt{W}$.
In terms of the transformed process $Y$, the Feynman-Kac formula becomes
\begin{align}
\exp(-t\mathcal{B})f(x)
=&
\mathbb{E}_x \left[\one_E \exp\left(  \frac{(a\beta)^2}{4}  \tau+  a\beta \int_{0}^\infty  L_Y(y,\tau)\circ dW(y)
\right) f\left(e^{-\beta Y_\tau}\right)\right] , \label{BroxSemigroup}
\end{align}
where $\tau=\tau(t)=\theta^{-1}(t)$.
Some interesting simplifications happen when calculating the functional in (\ref{BroxSemigroup}), for example
\begin{align}
&\int_{0}^t \frac{W'(X_u)}{\sqrt{X_u}}\, du 
= \int_{0}^{\theta^{-1}t} \frac{W'(X_{\theta(u)})}{\sqrt{X_{\theta(u)}}}  \beta^2 e^{-\beta Y_u}\, du \\
&\qquad = \int_{0}^{\theta^{-1}t}  \frac{dW}{dX} ( e^{-\beta Y_u})   \beta^2 e^{-\beta Y_{u}/2}\, du 
=\beta^2 \int_{0}^{\theta^{-1}t} \wt{W}'(Y_u)\, du.
\end{align}

One could now attempt to remove the white noise drift term from the SDE (\ref{BroxSDE}) using Girsonov's theorem. 
But this would transform the limiting functional $\Phi$ to a formal expression with a term $-\int_{0}^t W'(Y_u)^2\, du$, which must be interpreted as $-\infty$ for almost all sample paths $Y$.
This reflects the fact that the sample paths for Brox' diffusion are concentrated on a set (depending on $W$) which for almost all $W$ will have Wiener measure $0$ with respect to the ``quenched" measure on $B$ given $W$. 
Thus it does not appear feasible to modify the hard edge semigroup formula so that the driving process becomes standard Brownian motion.

\subsection{Rigorous interpretation of $X$ and $\Phi$} \label{SectionInterpretations}
We now rigorously define a stochastic process $X$ which formally satisfies the SDE (\ref{SDE1}).
Start with the following diffusion coefficient and potential function: 
\begin{align}
\sigma(x) =&\sqrt{2x} ,\qquad   
V(x)=  - \log(x)+\frac{2}{\sqrt{\beta}} \int_{x}^1 \frac{dW(y)}{\sqrt{y}} .
\end{align}
Let $B$ be a standard Brownian motion started at $s(X_0)$, and define $X$ by
\begin{alignat}{2} \label{speedScale1}
X_t=&s^{-1}(B_{T^{-1}t}),\qquad & 
s(x)=&-\int_{x}^1 e^{V(x)}\, dx,\qquad 
T^{-1}(t) =  \int_{0}^{t} \sigma(X_u)^2 e^{2V(X_u)}\, du.
\end{alignat}
Here $s$ and $T$ are respectively called the scale and time functions.  

Just before stating theorem \ref{mainThm}, we explained that throughout the paper our approximations will require a lower bound on the process $X$.
We now prove that $X$ satisfies such a lower bound with high probability.
\begin{lemma} \label{pathLowerBoundLemma}
For almost all $W$, all $x\in (0,1)$ and all $t>0$ we have
\begin{align}
P_{x}\left(\inf_{\tau\leq t} X(\tau) \leq \frac{1}{ \log n }  \;\; \mathrm{or}\;\; \sup_{\tau \leq t}X(\tau)\geq 1 \middle| W\right) \rightarrow 0\text{ as }n\rightarrow \infty.
\end{align}
\end{lemma}
\begin{proof}
It suffices to estimate the probability of hitting $1/\log n$ before $1$, so call this event $E$.
Since $
P(E|W) =  (s(1)-s(x))/(s(1)-s(1/\log n)),
$
and the numerator does not depend on $n$, it suffices to estimate the denominator.
\begin{align}
s(1)-s(1/\log n)=& -\int_{1/\log n}^1 y^{-1}e^{\frac{2}{\sqrt{\beta}}\wt{W}_{-\log y}}\, dy=  \int_{0}^{\log \log n}  e^{\frac{2}{\sqrt{\beta}}\wt{W}_z}\, dz. 
\end{align}
It is easy to see that the above integral will increase to $+\infty$ as $n\rightarrow \infty$ for almost all $W$.
\end{proof}

We now address the interpretation of the Stratonovich integral on line (\ref{SBO_semigroup}).
Because $X$ depends on the random environment $W$, the local time $L_{X}(x,t)$ will not be adapted to the filtration generated by $W$.  
Therefore the existence of our integral is not clear from the standard theory of stochastic calculus.
We assign a meaning to the integral following the approach of \cite{HLM16}.
Local time for $X$ can be expressed in terms of local time for $B$ by\footnote{
Derivation of the local time relation:
\begin{align}
\int_{0}^1 f(x)L_{X}(x,t)\, dx 
=& \int_{0}^t f(X_t)\, dt 
=\int_{0}^{T^{-1}(t)} f(s^{-1}(B_u)T'(u)\, du
=\int_{0}^{T^{-1}(t)} \frac{f(s^{-1}(B_u)}{(\sigma \cdot s')^2\circ s^{-1}(B_u)}\, du\\
=& \int_{0}^\infty \frac{f(b)L_B(b,T^{-1}(t))}{(\sigma \cdot s')^2\circ s^{-1}(b)}\, db 
=\int_{0}^\infty \frac{f(b)L_B(s(x),T^{-1}(t))}{(\sigma(x)^2 s'(x)}\, dx 
\end{align}
}
\begin{align}
L_X(x,t)=\frac{L_B(s(x),T^{-1}(t))}{\sigma(x)^2 s'(x)}.\label{changeLocTime}
\end{align}
The time $T^{-1}(t)$ is random and depends on $W$ in a complicated way, but we can replace it by a deterministic time, evaluate the integral and substitute the random time afterwards.
Note that on the right hand side of (\ref{changeLocTime}) the $x$ dependence is through $s(x)$ which only depends on the process $(W_y)_{y\geq x}$.
 Therefore with $T^{-1}(t)$ replaced by a deterministic time $u$, the local time is adapted to the filtration generated by $W$ \emph{in the $-x$ direction}.
For this reason we will work with ``anti-Ito" integrals\footnote{
Let $\Delta W_i=W_{(i+1)/n}-W_{i/n}$.  The different stochastic integrals can be summarized as follows:
\begin{align}
\int_{0}^1 Z_x \, dW_x \approx \sum_{i=0}^{n-1}Z_{i/n}\Delta W_i,\qquad
\int_{0}^1 Z_x \circ dW_x \approx \sum_{i=0}^{n-1}Z_{\frac{i+0.5}{n}}\Delta W_i,\qquad 
\int_{0}^1 Z_x \, d\ola{W}_x \approx \sum_{i=0}^{n-1}Z_{\frac{i+1}{n}}\Delta W_i .
\end{align}
}
throughout the paper, which we write using the notation $d\ola{W}$.
Finally, we can explain the meaning of the Stratonovich integral in equation (\ref{SBO_semigroup}):
\begin{align}
\int_{0}^\infty \frac{L_X(x,t)}{\sqrt{x}} \circ dW(x)
:=
\left.\int_{-\infty}^\infty \frac{L_B(s(x),u)}{\sqrt{x}\sigma(x)^2 s'(x)}   d\ola{W}(x) \right|_{u=T^{-1}(t)} 
- \int_{-\infty}^\infty \frac{L_B(s(x),T^{-1}(t))}{\sqrt{\beta}x\sigma(x)^2 s'(x)} \, dx \label{integral_interp}
\end{align}
The second integral in (\ref{integral_interp}) is the Ito-Stratonovich correction term.
It arises since $L_B(s(x),u)$ has no quadratic variation with respect to $W$, and because by Ito's lemma we have
\begin{align}
d s'(x) =&  \left(2/\beta -1\right) \frac{s'(x)}{x} \, dx - \frac{2s'(x)}{\sqrt{\beta x}} \,dW(x).
\end{align}


\subsection{Organization of the paper}

Our paper is organized as follows.
In section \ref{DefSection} we define a matrix model such that large powers of the matrix should converge to the stochastic Bessel semigroup operators.
For finite $n$ a discrete Feynman-Kac type formula holds exactly, and this leads us to define the processes $X^n$ and functionals $\Phi_n$.
Most of the paper is devoted to proving theorem \ref{mainThm}, but before undertaking this we discuss in section \ref{EdgeTransitionSection} the hard edge to soft edge transition from the semigroup point of view.
The functional $\Phi$ has two integrals, a fairly tame additive functional of the process $X$, and the much more ornery white noise local time integral.
The discrete functionals $\Phi_n$ we define in section \ref{phiNSection} decompose in the same way as $\Phi$, and so we divide our proof of theorem \ref{mainThm} into two parts.
First in section \ref{NoNoiseSection} we show convergence of the better behaved parts of the functionals which we call $\Phi^{A}_n$ and $\Phi^A$.
Our method is essentially to construct some approximate martingales and use them to bound a moment generating function.
In the remaining sections of the paper we carry out this method for the more difficult noisy parts of the functional.
This relies on a branching structure for the numbers of visits of the processes $X^n$ to different lattice points.
The crux of the proof is essentially to show that network of potential energy traps due to the random environment $W$ do not cause the discrete processes $X^n$ to get stuck in one place for too long.
Our argument for analyzing the noisy parts of the functionals roughly follows the template in section \ref{NoNoiseSection}, but is more complicated so we separate it into two parts, sections \ref{section5} and \ref{section6}.


\subsection{Acknowledgements}

It is a pleasure to thank  Brian Rider for many useful discussions.

\section{Random environment, driving process and matrix model} \label{DefSection}

In this section we define the random processes, functionals and matrix model that we will study in this paper. 
We do this in such a way that they are all coupled, and are only random through their dependence on two independent Brownian motions: the ``random environment" $W$, and another Brownian motion $B$ which we call the ``driving process".

At first it might appear simpler to work directly with the $\beta$-Laguerre ensembles.  
Our approach has two advantages: first that all of the $n\times n$ random matrices are naturally coupled to the limiting objects.  
This makes it easy to work in terms of almost sure convergence in many places.
The second advantage is that the matrix model we choose simplifies some parts of our analysis.
If one instead uses the $\beta$-Laguerre ensembles, then $\Phi_n$ must be modified, and will depend explicitly on the step types of $X^n$ (i.e.\ whether it jumps up, down or does not jump at each time point).
This makes the formulas more complicated, but in the end the same methods still work; so to streamline our exposition we prefer to avoid these complications.
Also if one works directly with the $\beta$-Laguerre ensembles, then the most convenient way to compare the processes $X^n$ and $X$ is to use the construction of Seignourel \cite{Seignourel00}.
There the scale function $s(x)$ is approximated using discrete versions $s_n(x)$, and estimating the associated error terms would add an additional section to the present paper.
Our construction of $X^n$ from $X$ is in fact a modified version of Seignourel's.

\subsection{Discrete process} \label{DiscreteProcessSection}
First we define lattice path processes $X^n$ that will approximate $X$.
Discretize space and time by
\begin{align}
x_k=k/n,\qquad  
t_i =i \Delta t,\qquad \Delta t =1/(4n^2),\qquad i,k\in \mathbb{N}\cup\{0\}.
\end{align}
We will use the notation $x_k$ very often.
By ``lattice path" we mean that $X^n$ is piecewise constant with jumps of only $\pm 1/n$ allowed only at the times $t_i$.
\begin{definition}
$X^n$ is stochastic process  defined by: $X^{n}$ is constant on any time interval not containing a time lattice point $t_i$, $i\in \mathbb{N}$; and $X^n(t_i)=X(\t_i)$, where $\t_i$ are stopping times defined below in display (\ref{gothicTDef}).
\end{definition}
Let $\t_1$ be the first time $X$ hits one of the lattice points $x_k$.
Let
\begin{align}
r_k=1-\frac{k}{2n}.
\end{align}
It will turn out that $r_k$ is the probability that the next step of $X^n$ is $\rightarrow$ given that the present location is $x_k$.
Suppose that the stopping times $\t_i$ have been determined for $i=1,\ldots ,\ell-1$ and let $x_k =X^{n}(\frac{ \ell-1}{n^2})$ be the present location of $X^n$.
Using the abbreviations $s_{k}=s(x_k)$ for values of the scale function,
choose $\alpha_+,\alpha_-$ so that 
\begin{align}
r_k =&\frac{s(\alpha_+)-s_k}{s_{k+1}-s_{k}} =\frac{s_k- s(\alpha_-) }{s_{k }-s_{k-1}} .
\end{align}
The next stopping time is then defined by
\begin{align} \label{gothicTDef}
 \t^{-}_{\ell} =&\inf\left\{ t \geq \t_{\ell-1}:\; 
nX_{t} \in \{ \alpha_-,\alpha_+ \} 
\right\}
\\
 \t_{\ell} =&\inf\left\{ t \geq \t^{-}_{\ell}:\; nX_{t} \in \{ k(\ell-1)-1  , k(\ell-1) , k(\ell-1)+1\} \right\},\qquad k(\ell)=nX_{t_\ell}.\nonumber
\end{align}
The above construction can be summarized as follows: the next jump of $X^n$ will be $\uparrow$ or $\downarrow$ depending on which of $x_{k-1}$ or $x_{k+1}$ the process $X$ hits next;
however, before this happens $X^n$ may take several $\rightarrow$ steps if $X$ ventures out of $(\alpha_-,\alpha_+)$ and then returns to $x_k$. 
As a general convention, we will not refer to $\rightarrow$ steps as jumps.


\subsection{Discrete functional} \label{phiNSection}
Formally applying the Feynman-Kac formula to the stochastic Bessel operator gives the limiting functional
\begin{align}
\Phi(X)= -\int_{0}^t \frac{a^2}{4X_s} + \frac{a W'(X_s)}{\sqrt{ \beta X_s}}\, ds.
\end{align}
To discretize this functional correctly we need to make some definitions. 
Let
$
x(i):=X^n(t_i)=X(\t_i),
$
and let $G_k$ and $G^{(2)}_k$ be random variables defined by 
\begin{align} \label{GDEF}
G_k =& -\sqrt{n}\int_{x_{k-1}}^{x_{k+1}} \left(\begin{cases}n( s-x_{k-1}) &\text{if }s<x_k \\
1-n(x_{k+1}-s) &\text{if } x_k <s \end{cases} \right)\, dW(s) ,\\
G^{(2)}_k =&  n \int_{x_{k-1}}^{x_{k+1}} \int_{x_{k-1}}^{s_2}\left(\begin{cases}
n(s_1 -x_{k-1})(1- n(s_2-x_{k-1})) & \text{if }s_1<s_2<x_k \\
0 & \text{if } s_1<x_k <s_2 \\
n(s_1 -x_k)(1 -n(s_2 -x_k)) & \text{if }x_k <s_1<s_2
\end{cases}\right)\, dW(s_1)\, dW(s_2).\nonumber
\end{align}
Note that $G_k$ and $G^{(2)}_k$ depend on $n$, and are scaled so that they will typically be of order $1$ as $n\rightarrow\infty$.
$G_k$ is Gaussian but $G^{(2)}_k$ is not.

The discretization of $\Phi$ that we will consider in this paper is
\begin{align} \label{phiFormula}
\Phi_n(t)=& 
\frac{-1}{4n^2} \sum_{i=1}^{\lfloor 4tn^2 \rfloor} 
  \frac{a^2}{4x(i)}+\frac{a\sqrt{n}G_{n x(i)}}{\sqrt{\beta x(i)}} +\frac{a  G^{(2)}_{nx(i)} }{\beta x(i)} .
\end{align}
At this point it is obscure why the third ``$G^{(2)}$ term" inside the sum should be included; the reason this term is included will be made clear in section \ref{discFKSection2}.

Let $p_k$ be the probability that the next jump of $X^n$ is $\uparrow$ given that the present location is $x_k$, and $q_k=1-p_k$.
The following lemma explains the origins of the definitions $G_k,G^{(2)}_k$.
Its proof is a straitforward but tedious calculation which we defer to appendix A.
%
\begin{lemma} \label{pkApproxLemma}
For almost all $W$ the approximation
\begin{align} \label{pkExpansion}
 \log \frac{p_k}{q_k}
=& \frac{2G_k}{\sqrt{\beta k} }+\frac{1}{k} +\frac{2G_{k}^{(2)}}{\beta k} + O\left( \frac{\sqrt{\log k}}{k^{3/2}}\right) 
\end{align}
holds for all $n\in \mathbb{N}$ and $k=1 ,\ldots , n$.
The implied constant depends only on $W,\beta$. 
\end{lemma}

\subsection{Matrix model} \label{matrixModelSection}

In this section we construct random matrix ensembles  $A=A^{(n)}$.
Start by naming the entries of a bidiagonal matrix as follows: for $k=1,\ldots ,n$ let $B_{k,k}= \mathbf{x}_k $ and $B_{k+1,k} =-\mathbf{y}_k$.
Now let
$A=\frac{1}{n}P B B^t P^{-1}$, where $P$ is the $n\times n$ diagonal matrix given by
\begin{align}
P_{k,k}= \prod_{j=1}^{k-1} \frac{\mathbf{y}_j}{\mathbf{x}_{j}}\sqrt{1+\frac{a}{j} }  . \label{basis3892}
\end{align}
Some motivation for this change of basis is explained in the footnote\footnote{
The definitions in this section imply the approximations 
\begin{align}
\mathbf{x}_{k}\approx &\sqrt{k} +\frac{G_k}{2\sqrt{\beta}}+\frac{a+1}{4\sqrt{k}} +\frac{G^{(2)}_k}{2\beta\sqrt{k}}-\frac{G_{k}^2}{8\beta\sqrt{k}},\qquad
\mathbf{y}_{k-1}\approx   \sqrt{k} -\frac{G_k}{2\sqrt{\beta}}-\frac{ a+1}{ 4\sqrt{k}}-\frac{G^{(2)}_k}{2\beta\sqrt{k}}-\frac{G_{k}^2}{8\beta\sqrt{k}} .
\end{align}
By a calculation as in \cite{ES07}, this suggests the limit $\sqrt{n} B_n \rightarrow \mathcal{L},$ where $\mathcal{L}= \sqrt{x}\frac{d}{dx} +\frac{W'(x)}{\sqrt{\beta}} + \frac{a+1}{2\sqrt{x}}$.
The expression $\mathcal{L}\mathcal{L}^t$ does not even make sense because it has a term $W'(x)^2/\beta$.
However, $\mathcal{L}= \sqrt{x}e^{I(x)} (\frac{d}{dx}+\frac{a+1}{2x})\circ e^{-I(x)}$ where $I(x)=\int_{x}^1 \frac{dW(y)}{\sqrt{\beta y}}$ and $\circ$ indicates composition of differential operators.
This leads us to the change of basis $e^{I(x)}\mathcal{L}\mathcal{L}^t e^{-I(x)}=\mathcal{B}$, and $\mathcal{B}$ turns out to be the operator in equation (\ref{SBO1}).
Using the approximations for $x_k$ and $y_{k-1}$, a formal calculation suggests
$
\prod_{k=\lfloor x n\rfloor}^n \frac{y_k}{x_k}\sqrt{1+a/k} \rightarrow e^{I(x)}, 
$
and thus the change of basis in equation (\ref{basis3892}) is a natural discretization of the conjugation leading to the operator $\mathcal{B}$.
}.

We have not yet defined $\mathbf{x}_k$ and $\mathbf{y}_k$, but clearly they will be related to $A$ by
\begin{align} \label{wtaEntries}
nA_{k,k-1}=- \mathbf{y}_{k-1}^2 \sqrt{1+\frac{a}{k-1}},\qquad nA_{k,k}= \mathbf{y}_{k-1}^2 + \mathbf{x}_{k}^2,\qquad 
n A_{k,k+1}=-\frac{\mathbf{x}_{k}^2 }{\sqrt{1+\frac{a}{k}}}.
\end{align}
Define $\mathbf{x}_k$ and $\mathbf{y}_k$ by
\begin{align} \label{HDef}
H_k =&  \frac{-1}{4n^2} \left(\frac{a^2}{4x_k}+ \frac{a\sqrt{n} G_{k}}{ \sqrt{\beta x_k} }+  \frac{ aG^{(2)}_k}{\beta x_{k}}  \right)\\
A_{k,k+1}=& -4 (1-r_k) p_k \exp(H_k), \qquad A_{k,k-1}= -4 (1-r_k) q_k \exp(H_k).  \label{q0091}
\end{align}
To be clear: substituting (\ref{wtaEntries}) in (\ref{q0091}) give formulas which we take as the definitions of $\mathbf{x}_k,\mathbf{y}_{k-1}$.
Recall that the $r_k$ and $p_k$ were defined in section \ref{DiscreteProcessSection} and that $q_k=1-p_k$.
This defines finishes the construction of our matrix model.

In the next two subsections we show that the matrices $A^n$ have the following properties:
\begin{proposition} \label{HardEdgeProp}
For each $k$, the first $k$ eigenvalues of $n^2 A^{(n)}$ converge in distribution to those of the stochastic Bessel operator.
\end{proposition}
\begin{proposition} \label{DiscFKProp}
There exists a functional $R(t)$ such that on the event $\ul{X}(t)\geq 1/\log n$ we have $|R(t)| \leq C(\log n)^3/\sqrt{n}$ such that the following discrete Feynman-Kac formula holds:
\begin{align}
((I-A^{(n)}/4)^m v)_k =\mathbb{E}_{x_k}\left[\one_E \exp( \Phi_n(t_m)+R(t_m))v(n X^n(t_m))\right]. \label{sg4}
\end{align}
where $v$ is any vector in $\mathbb{R}^n$ and $\Phi_n$ is the functional in equation (\ref{phiFormula}).
\end{proposition}
\textbf{Comment.} Since the eigenvalues of $A^{(n)}$ are of order $n^{-2}$ we will take $m=4tn^2$ in equation (\ref{sg4}), so that for any $k$ the first $k$ eigenvalues of the operator on the left hand side converge to $(\exp (-t \Lambda_i))_{i=1}^k$ where $\Lambda_i$ are the eigenvalues of the SBO.
Because the matrices $A^{(n)}$ are not symmetric, it is not obvious that the matrices on the left hand side of (\ref{sg4}) converge to $\exp (-t\mathcal{B})$; but it is at least clear that these are the right matrix powers to consider, if one hopes to extend the semigroup method to the random matrix hard edge.

\subsection{Discrete Feynman-Kac formula for $A^{(n)}$} \label{discFKSection2}

\begin{proof}[Proof of proposition \ref{DiscFKProp}]
Let $M=I-A^{(n)}/4$.  
It is well known that the entries of a power of a tridiagonal matrix can be expressed using a sum over lattice paths.
Let $\mathcal{P}^{m}_k$ be the set of functions $p:\;\{0,\ldots ,m\}\rightarrow \mathbb{Z}$ such that $p(0)=k$ and satisfying the conditions $p(i)-p(i-1) \in \{ -1,0,1\}$ for all $i=1,\ldots ,m$.
Let $E$ be the event that the range of $p$ is a subset of $\{1,\ldots ,n\}$.  
Then
\begin{align}
(M^m v)_k =\sum_{p \in \mathcal{P}^{m}_k} \one_E M_{p(0),p(1)} M_{p(1),p(2)}\ldots M_{p(m-1),p(m)}v_{p(m)} .
\end{align}
Letting $p(k) =n X^n(t_k)$, we can recognize the right hand side above as an expectation with respect to the process $X^n$ by 
factoring $M_{k,k+\Delta}= P(k,\Delta) e^{H(k,\Delta)}$ where $P(k,\Delta)$ is the probability that the next step of $X^n$ is an increment $\Delta/n$ given that the present location is $x_k$.
This defines the quantities $H(k,\Delta)$.  
We now have
\begin{align}
(M^mv)_k =\mathbb{E}_{x_k}\left[ \one_{E} v_{nX(t_m)}\exp \sum_{i=1}^m H(n X(t_i),\Delta X_i)\right] .
\end{align}
Therefore it suffices to show that $H(k,\Delta)=H(k)+O((\log(k) k^{-5/2})$ for all $\Delta$, where $H(k)$ is the function defined in equation (\ref{HDef}).
Actually, definitions (\ref{wtaEntries}-\ref{q0091}) are cooked up precisely so that
\begin{align}
e^{H(k,1)} = \frac{-A_{k,k+1}/4}{(1-r_k)p_k}=e^{H_k},\qquad 
e^{H(k,-1)} = \frac{-A_{k,k-1}/4}{(1-r_k)q_k}=e^{H_k},
\end{align}
where $H_k$ is defined in equation (\ref{HDef}).
Thus it remains only to calculate $H(k,0)$.
Equations (\ref{wtaEntries}-\ref{q0091}) give the relation
\begin{align}
r_k e^{H(k,0)} =1- \left( q_k/\sqrt{1+a/(k-1)} +p_k \sqrt{1+a/k}\right) (1-r_k)e^{H_k}. \label{aeorgibae}
\end{align}
A short calculation using formula (\ref{pkExpansion}) and Newton's binomial theorem gives
\begin{align}
q_k/\sqrt{1+a/(k-1)} +p_k \sqrt{1+a/k} =1- \frac{2n}{k} H_k +O\left(\frac{\sqrt{\log k}}{k^{5/2}}\right).
\end{align}
Substituting this in (\ref{aeorgibae}) yields $H(k,0)=H_k+O(\sqrt{\log k}/ k^{5/2})$.
\end{proof}


\subsection{Hard edge for the matrix model} \label{HardEdgeSection}

We now show that our matrix model has a stochastic Bessel operator limit.
By this we mean that for any $k$, the smallest $k$ eigenvalues of $A^{(n)}$ converge in joint distribution to those of the SBO.
In the case that the matrix entries are independent, the technique established in \cite{RR09} for proving convergence to the SBO is quite robust. 
Therefore we will just give the essential calculations here, and point the reader to \cite{RR09} for the technical details.

It is equivalent to show that the smallest $k$ singular values of $B=B^{(n)}$ converge in distribution to those of the operator
\begin{align}
\mathcal{L}=\sqrt{x} \frac{d}{dx} +\frac{a+1}{2\sqrt{x}} +\frac{1}{\sqrt{\beta}}W'(x).
\end{align}
The inverse of $\mathcal{L}$ is the following integral operator on $L^2[0,1]$:
\begin{align}
(\mathcal{L}^{-1}f)(x) =& \int_{0}^1 K(x,y)f(y)\, dy,\qquad K(x,y) =\one_{y\leq x} \frac{1}{\sqrt{x}} (x/y)^{a/2} \exp \int_{y}^x \frac{dW(z)}{\sqrt{\beta z}}
\end{align}
We check this by computing an integral kernel for $B^{-1}$, starting with the formula
\begin{align}
(B^{-1})_{i,j} =& \frac{1}{\mathbf{x}_j}\prod_{k=j+1}^{i} \frac{\mathbf{y}_{k-1}}{\mathbf{x}_k}.
\end{align}
Recall that $\mathbf{x}_k$ and $\mathbf{y}_k$ were defined in formulas (\ref{wtaEntries}-\ref{q0091}).
Unpacking these definitions, we get
\begin{align}
 \frac{\mathbf{y}_{k-1}^2}{\mathbf{x}_{k}^2}
=&
  \frac{A_{k,k-1}(1+a/(k-1))^{-1/2}}{A_{k,k+1}(1+a/k)^{1/2}} 
=   \frac{1}{   1+a/k } \frac{q_k  }{p_k  }(1+O(k^{-2}))  \nonumber \\
=&  \frac{1}{1+a/k} \frac{s_{k+1}-s_k}{s_{k}-s_{k-1}}(1+O(k^{-2})) .
\end{align}
Recall that $s_k=s(x_k)$ are values of the scale function for $X$.
For the scale ratio we make the following approximation:
\begin{align}
 \prod_{k=j+1}^i \frac{s_{k+1}-s_k}{s_{k}-s_{k-1}} 
= &\frac{s_{i+1}-s_i}{s_{j+1}-s_j}
=\exp\bigg(\int_{x_{j+1}}^{x_{i+1}} dV\bigg) \frac{  \int_{x_{i}}^{x_{i+1}} \exp(-\int_{y}^{x_{i+1}} dV)\, dy}
{ \int_{x_{j}}^{x_{j+1}} \exp(-\int_{y}^{x_{j+1}} dV)\, dy}  \nonumber \\
=&
\exp(V(x_i)-V(x_j)) (1+O(j^{-1/2}))  \nonumber \\
=& \frac{x_j}{x_i } \exp\left( \int_{x_j}^{x_i} \frac{2\,dW(y)}{\sqrt{\beta y}}\right)(1+O(j^{-1/2})).
\end{align}
Also we have
\begin{align}
\prod_{k=j+1}^i \frac{1}{1+a/k} =\exp  \sum_{k=j+1}^i \left(\frac{-a}{k} +O(k^{-2})\right) =& (x_{j}/x_{i})^{a}(1+O(j^{-1})).
\end{align}
Finally
\begin{align}
\mathbf{x}_{k}^2 =4(1-r_k )p_k \exp(H(k))\sqrt{1+a/k} = x_k (1+O(k^{-1/2})).
\end{align}
Our discussion so far shows that if $j\leq i$ then
\begin{align}
(B^{-1})_{ij}
=&  \frac{1}{\sqrt{x_i}} \left( x_{j}/x_{i} \right)^{a/2} \exp  \left( \int_{x_j}^{x_i} \frac{dW(y)}{\sqrt{\beta y}} \right)(1+O(k^{-1/2})).\label{kernel7}
\end{align}
Define a linear mapping of $\mathbb{R}^n$ into $L^2[0,1]$ sending the $k^{th}$ basis vector to $\one_{[x_{k-1},x_k]}/\sqrt{n}$. 
This identifies the matrices $B^{-1}$ with integral operators on $L^2[0,1]$.
The integral operators have piecewise constant kernel functions $K(x,y)$ given by equation (\ref{kernel7}) on the rectangle $[x_{i-1},x_i]\times [x_{j-1},x_j]$.
Thus the discrete integral kernels converge pointwise to the limiting kernel from \cite{RR09}.



\section{Hard edge to soft edge transition} \label{EdgeTransitionSection}

In this section we let
\begin{align}
\alpha=a/2.
\end{align}
Letting $\mathcal{B}_{a,\beta}(x)$ and $\mathcal{A}_\beta(x)$ respectively be the operators in equation (\ref{SBO1}) and (\ref{SAO1}), a formal calculation suggests that if one lets $x=1-\alpha^{-2/3}z$ then
\begin{align}
\alpha^{-4/3}\left( \mathcal{B}_{a,\beta}(x) -\alpha^2  \right)\rightarrow \mathcal{A}_{\beta}(z)\qquad \text{as}\qquad a\rightarrow \infty.
\end{align}
In particular one expects that under this scaling the smallest eigenvalue of the SBO should converge to that of the SAO.
This was in fact shown by Ramirez and Rider \cite{RR09}, up to a slightly different scaling which should not affect the limit. 
Using a Ricatti transformation, they showed that the eigenfunction equation for the SBO is equivalent to a diffusion process; they were able to show that this diffusion converges to one characterizing SAO eigenfunctions as $a\rightarrow \infty$.

It was shown in \cite{DTM16} that if $a=C\sqrt{n}$ then the smallest eigenvalue of the Laguerre unitary ensemble converges, properly scaled, to a Tracy-Widom distribution, i.e.\ the smallest eigenvalue of the SAO.
The authors of \cite{DTM16} rely on Riemann-Hilbert methods, and therefore their technique is specific to the classical parameter values $\beta=1,2,4$.

First in section \ref{LimitTransitionSection} we will formally show how the SBO semigroup degenerates to the SAO semigroup.
Within the method of \cite{RR09} it is not convenient to let $a$ depend on $n$, so it would be interesting to obtain a result describing how $a$ can depend on $n$ if one is to observe an SBO to SAO transition.
In section \ref{FiniteNTransitionSection} we explore this point from the semigroup perspective.

\subsection{Transition of limiting Feynman-Kac formulas} \label{LimitTransitionSection}

We now give some partial results which suggest how the edge transition problem can be approached from the semigroup point of view.
First decompose the SBO as a generator plus source term
\begin{align}
 -\mathcal{B} =& \frac{1}{2}\sigma^2 \partial_{x}^2 +\mu \partial_x - K;\\
  \sigma=&\sqrt{2x}, \qquad \mu= 1+2\sqrt{ \frac{x}{\beta}}W'(x) ,\qquad K= -\frac{\alpha^2}{ x}- \frac{2\alpha W'(x)}{\sqrt{\beta x}}.
\end{align}
Assuming that the Feynman-Kac formula holds for the SBO, we have
\begin{align}
&\exp\left(  -t \alpha^{-4/3} (B-\alpha^2 I) \right)f(\alpha^{2/3}(x-1)) \label{transition1} \\
&\hspace{70pt}=\mathbb{E}_x \exp\bigg(-\int_{0}^{\alpha^{-4/3}t}  K(X_u)+\alpha^2\, du \bigg) f(\alpha^{2/3}(X_{\alpha^{-4/3}t}-1)) \nonumber.
\end{align}
For each $a>0$, define new stochastic processes $\wt{B}$ and $w$ by
\begin{align}
\wt{B}(t) =- \alpha^{2/3} B(2\alpha^{-4/3} t), \qquad w(z) =\alpha^{1/3}W(1-\alpha^{-2/3}z). \label{ScaledProcessDef}
\end{align}
Note that $\wt{B}$ is a Brownian motion with a diffusion coefficient of $\sqrt{2}$.
We now state a result which suggests that the right hand side of (\ref{transition1}) should converge to $\exp(-t\mathcal{A} )f(x)$ as $a\rightarrow \infty$, where $\mathcal{A}$ is the SAO.
Note that the authors of \cite{GS16} consider $\exp(-t\mathcal{A}/2)$, so our formulas will differ by some powers of $2$.
\begin{proposition}\label{SemigroupConvProp}
Suppose $X_0=1-\alpha^{-2/3}z$ where $|z|\leq   \log \alpha $, then
\begin{align}
\left| \alpha^{2/3}(1-X_{\alpha^{-4/3}t}) -   \wt{B}_t  \right| \leq  & C a^{-1/3}(\log a)^2 \\
\int_{0}^{\alpha^{-4/3}t} K(X_u)+\alpha^2\, du \rightarrow & \int_{0}^t  \wt{B}_u \, du+\frac{2}{\sqrt{\beta}}\int_{0}^\infty L_{\wt{B}}(x,t)\, dw(x)
\end{align}
\end{proposition}
To prove the above proposition we will use the following two lemmas.
\begin{lemma} \label{ShortTimeLemma}
For almost all $W$ there exists a finite constant $C=C(W,\beta)$ such that if $|z|\leq \log \alpha$ then 
\begin{align}
\left| \alpha^{2/3}s(1-\alpha^{-2/3}z) - z \right| \leq C  \alpha^{-1/3} (\log \alpha)^2 .
\end{align}
For almost all $W,B$ there exists a finite constant $C=C(W,B,\beta)$ such that if $|Z_0 |\leq  \log \alpha$ and $T\leq \log \alpha$, then
\begin{align}
\left| \alpha^{4/3}T^{-1}(\alpha^{-4/3} t) -2  t\right| \leq & C  \alpha^{-1/3} (\log a)^2.
\end{align}
\end{lemma}
The proof of lemma \ref{ShortTimeLemma} is a direct but tedious calculation, so we will postpone it until the end of the section.

We will need some continuity estimates for Brownian local time.  
It is easy to find a variety of very sharp estimates in the literature, for example \cite{BorodinSurvey}, but not exactly in forms that are convenient for us.
However for our purposes it is not important to have an optimal modulus of continuity, and therefore the lemma we need is accessible using the method developed in \cite{McKeanLocalTime}.
Since the method is relatively old and well understood, and since the following is not a very ambitious statement, we will not give a proof.
\begin{lemma} \label{LocalTimeLemma}
For almost all $B$ the following quantities are finite:
\begin{align}
\sup_{t \geq 0} \sup_{\substack{x,y\in \mathbb{R}\\ |x-y|<1}} \frac{|L_B(x,t)-L_B(y,t)|}{1 \vee t^{0.251} |x-y|^{0.499}},\qquad 
\sup_{\substack{t,t'\geq 0\\ |t-t'|<1}} \sup_{x\in \mathbb{R}} \frac{|L_B(x,t)-L_B(x,t')|}{1 \vee t^{0.01} |t-t'|^{0.499}}.
\end{align}
\end{lemma}

\begin{proof}[Proof of proposition \ref{SemigroupConvProp}]
Throughout the proof we will see a proliferation of error terms that are bounded in absolute value by expressions of the form $C(B,W,\beta) \alpha^{p}(\log \alpha)^{q}$, which we will abbreviate using the notation $O(\alpha^{p})^-$.
First we prove that the rescaled process is approximately a Brownian motion.
Our convention is that $B$ starts at $X_0$. 
But to understand how the starting position affects our calculation, just for now we let $X_t=s^{-1}(X_0 +B_{T^{-1}t})$ where $B$ starts at $0$.
\begin{align}
\alpha^{2/3}\left(1-X_{\alpha^{-4/3}t} \right) 
=& \alpha^{2/3} \left(1- s^{-1}(s(X_0)+ B_{T^{-1}(\alpha^{-4/3}t)}\right) \\
= & \alpha^{2/3}\left(1-s^{-1} \left( \alpha^{-2/3}z +O(a^{-1})^- +B_{2 \alpha^{-4/3}t +O(\alpha^{-5/3})^-}\right)\right)\nonumber  \\
=&  \alpha^{2/3}\left(1-s^{-1} \left( \alpha^{-2/3} z+ B_{2 \alpha^{-4/3}t } + O(\alpha^{-5/6} )^-\right) \right)\label{lil111} \\
=&  \alpha^{2/3}  \left(1 -\left(1-\alpha^{-2/3}z -\alpha^{-2/3} \wt{B}_{t } + O(\alpha^{-1})^-\right)\right) .\nonumber 
\end{align}
For line (\ref{lil111}) we used the law of the iterated logarithm.  For the other steps we used lemma \ref{ShortTimeLemma}.

Now we turn our attention to the convergence of the functionals.
The first term is easy:
\begin{align}
   &\int_{0}^{\alpha^{-4/3}t} \frac{\alpha^2 }{X_u}-\alpha^2 \, du 
=  \alpha^2\int_{0}^{\alpha^{-4/3}t}  \frac{1}{1-\alpha^{-2/3}( \wt{B}_{\alpha^{4/3}u}+O(\alpha^{-1/3})^-  )}-1 \, du \\
&\qquad =    \alpha^{2} \int_{0}^{\alpha^{-4/3}t} \alpha^{-2/3}\wt{B}_{\alpha^{4/3}u}  \, du +O(\alpha^{-1/3})^- 
= \int_{0}^t \wt{B}_u \, du +O(\alpha^{-1/3})^- .
\end{align}
Now for white noise integral term.
Using the interpretation defined in the introduction, we have
\begin{align}
\int_{0}^{\alpha^{-4/3}t}  \frac{W'(X_u)}{\sqrt{X_u}}\, du 
:=&\int_{0}^1 \frac{L_B(s(x),T^{-1}(\alpha^{-4/3}t))}{2\sqrt{x}e^{I(x)}} d\ola{W}(x) +\int_{0}^1 \frac{L_B(s(x),T^{-1}(\alpha^{-4/3}t))}{2\sqrt{\beta }x e^{I(x)}} dx. 
\end{align}
We now change variables by $x=1-\alpha^{-2/3}z$.  
Since $W(x) =\alpha^{-1/3} w(\alpha^{2/3}(1-x))$, the Jacobian factor is $\alpha^{-1/3}$ in the stochastic integral, but $\alpha^{-2/3}$ in the ordinary integral.
Therefore we can ignore the ordinary integral.
\begin{align}
 \alpha\int_{0}^1 \frac{L_B(s(x),T^{-1}(\alpha^{-4/3}t))}{2\sqrt{x}e^{I(x)}} d\ola{W}(x)  
  =  \alpha^{2/3}\int_{0}^{\alpha^{2/3}} \frac{L_B(s(1-\alpha^{-2/3} z),T^{-1}(\alpha^{-4/3}t))}{2\sqrt{1-\alpha^{-2/3}z}e^{I(1-\alpha^{2/3}z)}} d\ola{w}(z) \label{aeorug}
\end{align}
We now approximate the local time using lemmas \ref{ShortTimeLemma} and \ref{LocalTimeLemma}:
\begin{align}
L_B(s(1-\alpha^{-2/3} z),T^{-1}(\alpha^{-4/3}t))
=& L_B( \alpha^{-2/3} z +O(\alpha^{-1})^-,2\alpha^{-4/3}t+O(\alpha^{-5/3})^-)\\
=&2\alpha^{-2/3}L_{\wt{B}}(  z +O(\alpha^{-1/3})^-,  t+O(\alpha^{-1/3})^- ) \\
=& 2\alpha^{-2/3}L_{\wt{B}}(   z  ,  t )+O(\alpha^{-1/6})^-.
\end{align}
Applying the above result in the integral (\ref{aeorug}) we see that it equals
\begin{align}
\alpha^{2/3}\int_{0}^{\alpha^{2/3}} \frac{ 2\alpha^{-2/3} L_{\wt{B}}(   z ,  t)}{2 +O(\alpha^{-1/3} )^-} d\ola{w}(z) +O(\alpha^{-1/6})^-  
=
 \int_{0}^{\infty}  L_{\wt{B}}(  z  ,  t  )\, dw(z)  +O(\alpha^{-1/6})^- . 
\end{align}
Since $L_{\wt{B}}(z,t)$ does not depend on $w$, it does not matter whether we write the last integral as Ito or anti-Ito.
Clearly we incur a negligible error by extending the domain of integration to infinity.
\end{proof}


\subsection{Edge transition with $a$ depending on $n$}\label{FiniteNTransitionSection}

We begin by decomposing our functionals as follows:
\begin{align} \label{ABDecomposition}
\Phi_n =&  a^2\Phi^{A}_n + a \Phi^{B}_n, \qquad 
\Phi^{A}_n =  
\frac{-1}{4n^2} \sum_{i=1}^{\lfloor 4tn^2 \rfloor} 
\frac{1}{4x(i)},\qquad
\Phi^{B}_n=
\frac{-1}{4n^2} \sum_{i=1}^{\lfloor 4tn^2 \rfloor} 
 \frac{ \sqrt{n}G_{n x(i)}}{\sqrt{\beta x(i)}} +\frac{   G^{(2)}_{nx(i)} }{\beta x(i)}  \\
\Phi  =&  a^2\Phi^{A}  + a \Phi^{B} , \qquad 
\Phi^{A} = -\int_{0}^t \frac{1}{4X_s} ,\qquad
\Phi^B =  -\int_{0}^t \frac{  W'(X_s)}{\sqrt{ \beta X_s}}\, ds.
\end{align}
Also let $\Psi$ be the limiting functional for the SAO semigroup:
\begin{align}
\Psi =-\int_{0}^t \wt{B}_u \, du - \frac{2}{\sqrt{\beta}}\int_{0}^{\infty } L_{\wt{B}}(x,t) \, dw(x),
\end{align}
where $\wt{B}$ and $w$ were defined in display (\ref{ScaledProcessDef}).

In the following proposition we write our functionals with a second argument to indicate the starting position of the process.
So $\Phi(t,x)$ means the functional $\Phi$ integrating to time $t$, with the process $X$ started at $x$.
Also we use the notation $\ul{X}(t)=\inf_{u\leq t}X_u$.
\begin{proposition}
Allow $a$ to depend on $n$ in such a way that $a(n)=o(n^{0.249})$. 
For almost all $B,W$ and any $z$, on the event $\ul{X}(\alpha^{-4/3}t)\geq 1/\log n$ we then have
\begin{align} 
\left| \Phi_n(\alpha^{-4/3}t, 1-\alpha^{-2/3}z) -\Psi(t,z) \right| \rightarrow 0 \qquad \text{as}\qquad n \rightarrow \infty.
\end{align}
\end{proposition}
Note that the hypothesis on $\ul{X}$ will be satisfied almost surely if $a(n)\geq n^p$ for some $p>0$.
\begin{proof}
Our assertion is basically a corollary of the main results of later sections of the paper.
Suppressing the arguments of the functionals we estimate
\begin{align}
\left|\Phi_n -\Psi \right| \leq  &  a(n)^2 \left| \Phi^{A}_n -\Phi^A \right| +a(n)\left| \Phi^{B}_n -\Phi^{B}\right| +\left| \Phi -\Psi \right|.
\end{align}
In proposition \ref{SemigroupConvProp} we showed that the third term converges to zero.
Proposition \ref{ATermsProposition} asserts that if $\ul{X}(\alpha^{-4/3}t)\geq 1/\log n$, then $\Phi^{A}_n -\Phi^A =O(n^{-0.499})$.
Thus $a(n)^2 \left|\Phi^{A}_n -\Phi^A\right|$ converges to zero.
Applying propositions \ref{PhiNPhiTildeThm} and \ref{WtPhiConverges}, we have $\Phi^{B}_n -\Phi^B =O(n^{-0.249})$ and thus $a(n)\left|\Phi^{B}_n -\Phi^B\right|$ converges to zero.
\end{proof}

\subsection{Proof of lemma \ref{ShortTimeLemma}}
\begin{proof}[Proof of lemma \ref{ShortTimeLemma}]
We use the representation
\begin{align}
X_t=& s^{-1}(B_{T^{-1}(t)}), \qquad s(x)=-\int_{x}^1  e^{V(y)}\, dy ,\qquad T(t)=\int_{0}^t \frac{du}{(s'\cdot \sigma)^2\circ s^{-1}(B_u)},
\end{align}
where $B$ is a Wiener process started from $s(X_0)$.
By the reflection principle $\sup_{t\leq \alpha^{-4/3}\log \alpha} B_t -B_0$ is normal with mean zero and variance $\alpha^{-4/3}\log \alpha$, and so by a standard tail bound and the Borel-Cantelli lemma
\begin{align}
\sup_{t\leq   \alpha^{-4/3}\log \alpha} |B_t -B_0|\leq   2\alpha^{-2/3}\log \alpha
\end{align} 
for all sufficiently large $\alpha$.

Recall that
\begin{align}
s'(x) =x^{-1} e^{I(x)},\qquad \text{where}\qquad I(x)=\int_{x}^1 \frac{2 dW(y)}{\sqrt{\beta y}} =\frac{2}{\sqrt{\beta}} \wt{W}_{-\log x} ,\label{sPrimeFormula}
\end{align}
and $\wt{W}$ is another Wiener process.
Using once again that the running maximum of $\wt{W}_\xi$ is a normal random variable, for sufficiently large $\alpha$ we have
\begin{align}
\sup_{ x\in [1-\alpha^{-2/3}\log \alpha,1]} \wt{W}_{-\log x} 
\leq \sup_{ \xi\leq 2 \alpha^{-2/3}\log \alpha} \wt{W}_{\xi }  
\leq  2\alpha^{-1/3}\log \alpha\qquad \text{a.s.}
\end{align}
Taylor expanding the exponential in (\ref{sPrimeFormula}) and using the above estimate, we get the bound 
\begin{align}
|s'(x) -1|\leq  \alpha^{-1/3}\log \alpha \qquad \text{for all }x\in [1-\alpha^{-2/3}\log \alpha,1]\text{ if $\alpha$ is large enough}.
\end{align}
The bounds we asserted on the scale function and its inverse then follow easily.

We now approximate the derivative of the time change function.
For $u\leq \alpha^{-4/3} \log \alpha$ we have so far shown that
\begin{align}
|B_u -B_0|\leq \alpha^{-2/3}\log \alpha ,\qquad  s^{-1}(B_u)-1 =B_u-B_0 +O(\alpha^{-1}(\log \alpha)^2) =O(\alpha^{-2/3}\log \alpha).
\end{align}
Using $\sigma(x)=\sqrt{2x}$ and the above approximation for $s'$, it follows that if $a$ is sufficiently large then for all $t\leq \alpha^{-4/3}\log \alpha$ we have
\begin{align}
T'(t)= \frac{1}{2} +O(\alpha^{-1/3}\log \alpha).
\end{align} 
We conclude that
\begin{align}
T^{-1}(t) -2t =O(a^{-5/3}(\log \alpha)^2)\qquad \text{for }t\leq \alpha^{-4/3}\log \alpha.
\end{align}
\end{proof}


\section{Convergence of the noiseless terms} \label{NoNoiseSection}

Recall that in display (\ref{ABDecomposition}) we decomposed our functionals as
\begin{align}
\Phi_n =&  a^2\Phi^{A}_n + a \Phi^{B}_n, \qquad \Phi  =   a^2\Phi^{A}  + a \Phi^{B} ,
\end{align}
where the ``$A$" and ``$B$" terms do not depend on $a$.
We now further divide the task of comparing $\Phi_n$ and $\Phi$ into two steps as follows:
\begin{align} \label{MakePhiTilde}
|\Phi -\Phi_n| \leq &  |\Phi -\wt{\Phi}_n| + | \wt{\Phi}_n-\Phi_n|, \qquad \text{where } \wt{\Phi}_n = a^2 \wt{\Phi}^{A}_n +a \wt{\Phi}^{B}_n \\
\wt{\Phi}^{A}_n (t) =& -\sum_{k=1}^n    \frac{1}{4 x_k}L_{X}([x_k,x_{k+1}) ,t),\\
\wt{\Phi}^{B}_n (t) =&  -\sum_{k=1}^n  \left(\frac{  \sqrt{n} (G_{k+1}+G_k)}{2\sqrt{\beta x_k}} +\frac{   G^{(2)}_k}{ \beta x_k}   \right) L_{X}([x_k,x_{k+1}) ,t).
\end{align}
It is much more difficult to analyze the ``$B$" terms (the white noise/ local time integral and its discretization).
In this section we carry out the easier task of analyzing the noiseless ``$A$" terms; we will prove the following.
\begin{proposition} \label{ATermsProposition}
On the event $\ul{X}(t)\geq 1/\log n$ we have
\begin{align}
|\Phi^{A}_n -\Phi^A | \leq C t n^{-0.499}.
\end{align}
This holds for almost all $B,W$, and $C=C(B,W,\beta)$.
\end{proposition}


\subsection{Estimating $|\wt{\Phi}^{A}_n -\Phi^{A}_n|$} \label{SectionAAN}

Recall that the steps of $X^n$ are equal to increments of $X$ between the stopping times $\t_i$ constructed in section \ref{DiscreteProcessSection}.
However, it will sometimes be convenient to only consider times at which $X^n$ jumps (i.e.\ not a $\rightarrow$ step).
Therefore define a sequence of times $\tau_i$ by letting $\tau_0=0$ and 
\begin{align}
\tau_{i}=\inf\{ t_j\geq \tau_{i-1}:\; X^{n}(t_{j})\neq X^{n}(t_{j-1})\}.
\end{align}
Also, we define another sequence of times $\wt{\tau}_i$ by
\begin{align}
\wt{\tau}_i=\t_j \qquad \text{if and only if}\qquad \tau_i=t_j.
\end{align}
So $\tau_i$ are the jump times for $X^n$, and $\wt{\tau}_i$ are the jump times for $X$.
We also introduce a notation for the sequence of lattice points visited after each jump:
\begin{align}
x(i) =X^n(\tau_i) =X(\wt{\tau}_i).
\end{align}
Notice that increments of $\Phi_n$ depend on the current position of $X^n$, which is the most recent lattice point visited by $X$; whereas $\wt{\Phi}_n$ depends on which interval $[x_k,x_{k+1})$ the process $X^n$ lies in.
Therefore during each jump it is important to distinguish between time spent by $X$ above and below the lattice point $x_k$, so we define
\begin{align}
\Delta\tau_i=&\tau_{i+1}-\tau_{i},\qquad \Delta \wt{\tau}_i = \wt{\tau}_{i+1}-\wt{\tau}_i, \\ 
\Delta \wt{\tau}^{\uparrow}_i =&\text{meas}\,\{t\in [\wt{\tau}_{i},\wt{\tau}_{i+1}):\; X\geq x(i),\\
\Delta \wt{\tau}^{\downarrow}_i =&\text{meas}\,\{t\in [\wt{\tau}_{i},\wt{\tau}_{i+1}):\; X< x(i).
\end{align}
Here ``meas" is Lebesgue measure.
With all this notation in hand, we can state the following formula which is clear from the definitions:
\begin{align}
\Phi^{A}_n(\tau_j) - \wt{\Phi}^{A}_n(\wt{\tau}_j) =& \sum_{i=1}^{j}\xi_i ,\qquad \xi_i= \frac{\Delta\tau_i -\Delta\wt{\tau}_i}{x(i)} + \Delta\wt{\tau}^{\downarrow}_i \left(\frac{1}{x(i)} -\frac{1}{x(i)-1/n}\right)  .
\end{align}
Since we have evaluated the functionals at different times, we take on the burden of also bounding 
$|\Phi^{A}_n(\tau_j) - \Phi^{A}_n(\wt{\tau}_j)|$, but this will be relatively easy.

The first step is to compute the means of the time increments.
The approximation we prove here is stronger than we need for this section but will be necessary later.
Define the following variables:
\begin{align}
\gamma_k=&  4n^{5/2} \int_{x_k}^{x_{k+1}} (x-x_k)(x_{k+1}-x)\, dW(x).
\end{align}
Note that with the explicit factor of $n^{5/2}$ the $\gamma_k$'s are typically order 1.
\begin{lemma} \label{exitTimeLemma01}
For almost all $W$, all $n\in \mathbb{N}$ and all $k=\frac{n}{\log n},\ldots ,n$ we have
\begin{align}
\mathbb{E} \left[ \Delta \wt{\tau}^{\uparrow}_i \middle| W,x(i)=x_k \right] =&\frac{1}{x_k}\left( 1 +\frac{G_k+\gamma_k}{\sqrt{\beta k}}\right)\Delta t+O(n^{-3}\log (n)^3), \\
\mathbb{E} \left[ \Delta \wt{\tau}^{\downarrow}_i \middle| W,x(i)=x_k \right] =&\frac{1}{x_k}\left( 1 -\frac{G_k+\gamma_{k-1}}{\sqrt{\beta k}}\right)\Delta t+O(n^{-3}\log (n)^3).
\end{align}
The implied constants depend only on $W$ and $\beta$.
\end{lemma}
\begin{proof}
By the general theory of one dimensional diffusion process, if the process starts at $x_k$, then the expected time the process spends in any subinterval $A\subset [x_{k-1} ,x_{k+1}]$ before the first exit is
\begin{align}
\mathbb{E}[T(A)] =& \int_A  \frac{ 2(s(y)\wedge s_k  -s_{k-1})(s_{k+1}  - s (y) \vee s_k)}{(s_{k+1}-s_{k-1})s'(y) \sigma(y)^2} \, dy.
\end{align}
Here $\wedge,\vee$ precede $+,-$ in order of operations.
Letting $A=[x_k,x_{k+1}]$ and substituting some of the definitions given in displays (\ref{speedScale1}) and (\ref{pkExpansion}), we get
\begin{align}
\mathbb{E}[\Delta\wt{\tau}^{\uparrow}_i|W,x(i)=x_k]
=&
  p_k \int_{x_k}^{x_{k+1}} \int_{y}^{x_{k+1}} \exp\left( \int_{y}^z dV\right) y^{-1} \,dy.
\end{align}
For the quality of approximation claimed in the lemma statement, it suffices to substitute a linear approximation for the exponential function, and $x_{k}^{-1}$ for the explicit factor of $y^{-1}$.
After a straitforward calculation one obtains the first formula in the lemma statement, and the second formula is proven similarly.
\end{proof}

\begin{lemma} \label{mgfLemma01}
Let $M(\lambda) =
  \mathbb{E}[  \exp(\lambda   \xi_i ) |x(i)=x_k ]$.
For all $n\in \mathbb{N}$ and $k=\frac{n}{\log n},\ldots n$ we have
\begin{align}
\left| M(\lambda)-1-\lambda M'(0) \right|
\leq    \frac{C\lambda^2 (\log n)^2}{k^{4}} \qquad \text{for all }\lambda \leq  \frac{k^{2}}{C\log n} .
 \label{taylorError1}
\end{align}
\end{lemma}
\begin{proof}
Suppose a random variable $Q$ satisfies a tail bound \begin{align}
P(Q\geq r)\leq C\exp(-c r) \label{tb1}.
\end{align}
Then for $\lambda< c/2$ the second derivative of the moment generating function for $Q$ can be bounded by
\begin{align}
M_{Q}''(\lambda) =&\int_{0}^\infty (2q+\lambda q^2)e^{\lambda q} P(Q>q)\, dq 
\leq\int_{0}^\infty (2q+\frac{c}{2} q^2)Ce^{-c q/2}\, dq 
=  24C /c^2. \label{TailToMGF}
\end{align}

We will apply the above bound to $Q=k^{2}\xi_i /\log n$.
Since we assumed that $k\geq n/\log n$ we have $1/x(i)\leq \log n$.
After conditioning on $x(i)$, $\xi_i$ is random only through $\Delta \tau_i , \Delta \wt{\tau}_{i}^{\uparrow},\Delta \tau_{i}^{\downarrow}$.
Since the diffusion coefficient for $X$ is bounded below by $\sqrt{(k-1)/n}$ until it leaves $(x_{k-1},x_{k+1})$, it is clear that the probability of the process exiting in the next $1/(kn)$ time units is bounded below by a constant regardless of the position of $X$ within $[x_{k-1},x_{k+1}]$. 
At this point it is clear that $k^{2}\xi_i /\log n$ satisfies (\ref{tb1}) with constants $C,c$ depending only on $\beta$,
so the lemma statement follows from Taylor's theorem with the remainder in Lagrange form.
\end{proof}


\begin{lemma} \label{HMartProp01}
On the event $\ul{x}(4tn^2)\geq 1/\log n$ we have, for all $n\in \mathbb{N}$ and all positive integers $m\leq 4tn^2$ that
\begin{align}
\left|\Phi^{A}_n(\tau_m) -\wt{\Phi}^{A}_n(\wt{\tau}_m) \right| \leq C t n^{-0.499}   .
\end{align}
The above holds for almost all $B,W$, and $C=C(B,W,\beta )$.
\end{lemma}
\begin{proof}
Given $x(i)$, each $\Delta \tau_i /\Delta t$ is a geometric random variable with parameter $x(i)/2$.
From this and lemma \ref{exitTimeLemma01}, it is easy to see that 
\begin{align}
\left|\mathbb{E}\left[ \xi_i \middle| x(i)\right] \right| \leq C \log( n)  (nx(i))^{-5/2},
\end{align}
where $C$ depends on $W,\beta$.
The factor of $\log n$ came from using elementary techniques to get an almost sure bound on the gaussian random variables $G_k +\gamma_k$ in lemma \ref{exitTimeLemma01}.
For $x(i)\geq 1/\log n$ using this mean estimate in lemma \ref{mgfLemma01} gives
\begin{align}
\mathbb{E}[\exp(\lambda \xi_i)|F_i] =& 1 + \lambda O( (\log n)^p n^{-5/2} ) +\lambda^2 O((\log n)^p n^{-4}) 
\end{align}
for all $\lambda $ of order $n^2/ (\log n)^p $.
Here $p$ is a universal constant which in this case can be taken to be 3.5, but throughout the paper we will use this notation to avoid keeping track of irrelevant powers of $\log n$.

Define the following notations:
\begin{align}  
\ul{x}(i) =\min_{j=1}^i x(j), \qquad Z_i(\lambda) =\exp\bigg( \lambda \sum_{j=1}^i  \xi_j  \bigg),
\end{align}
and let $F_i$ be the $\sigma$-field generated by $\{W,x(1),\ldots ,x(i) \}$.
Let $E(i)$ be the event where $ \ul{x}(i)\geq 1/\log n$, and calculate
\begin{align} \label{MartCalc1}
\mathbb{E}[\one_{E(j)}Z_j|W   ]  
 \leq & \mathbb{E}[\one_{E(j)}Z_{j-1}  \mathbb{E}[ \exp  (\lambda   \xi_j  )|F_j ]|W ] \\
 \leq & 
\mathbb{E}\left[ \one_{E (j)}Z_{j-1}   \middle|W \right]\left(1+C  (\log n)^p \left(\lambda n^{-5/2} +\lambda^2  n^{-4}\right) \right) \label{MartCalc2a}
\end{align}
By induction on $j$, the inequality $e^x\geq 1+x$, and $m\leq 4tn^2$ it follows that
\begin{align}\label{mgfBound1}
\E[\one_{ E(i)}Z_i|W]
\leq    \exp\left( C t  (\log n)^p \left(\lambda n^{-1/2} +\lambda^2  n^{-2}\right)\right). 
\end{align}
By Chebychev's inequality we get
\begin{align}
&\log P\left(\Phi^{A}_n(\tau_i) -\wt{\Phi}^{A}_n(\wt{\tau}_i)  \geq tn^{-0.499} ,\ul{x}(i)\geq \log(n)^{-1}\middle| W \right)\\
& \leq  -\lambda t n^{-0.499} + \log \E[\one_{ E(i)}Z_i|W].
\end{align}
Now choose $\lambda =\sqrt{n}/t$ and bound the expectation using (\ref{mgfBound1}).
By making a similar argument with $Z_i(\lambda)$ replaced by $Z_i(\lambda)^{-1}$ we get the reverse inequality, thus the absolute value in the inequality
\begin{align}
\log P\left(\Phi^{A}_n(\tau_i) -\wt{\Phi}^{A}_n(\wt{\tau}_i)  \geq t n^{-0.499} ,\ul{x}(i)\geq \log(n)^{-1}\middle| W \right)\leq  -c n^{0.001}.
\end{align}
Applying the Borel-Cantelli lemma finishes the proof.  
\end{proof}

\begin{lemma}
On the event $\ul{X}(1.01t)\geq 1/\log n$ we have, for all $n\in \mathbb{N}$ that
\begin{align}
\left|\Phi^{A}_n(t) -\wt{\Phi}^{A}_n(t) \right| \leq C t n^{-0.499}   .
\end{align}
The above holds for almost all $B,W$, and $C=C(B,W,\beta )$.
\end{lemma}
\begin{proof}
First we claim for almost all $B,W$ and all $i\leq 4t n^2$ that $|\tau_i -\wt{\tau}_i|\leq C tn^{-0.499}$, with the constant depending only on $B,W,\beta$.
This can be proven in exactly the same way as lemma \ref{HMartProp01}, so we will not give the details.
Let $I$ be the smallest integer such that $\tau_I\geq t$.  
Since the times $\tau_i$ are a subset of the time lattice points $\mathbb{Z}/(4n^2)$, it follows that $I \leq 4tn^2$.
We then estimate as follows:
\begin{align} \label{pfEq938}
\left|\Phi^{A}_n(t) -\wt{\Phi}^{A}_n(t) \right| \leq &
\left|\Phi^{A}_n(t) - \Phi^{A}_n(\tau_I) \right|
+ \left|\Phi^{A}_n(\tau_I) -\wt{\Phi}^{A}_n(\wt{\tau}_I) \right|
+ \left|\wt{\Phi}^{A}_n(\wt{\tau}_I) -\wt{\Phi}^{A}_n(t) \right|.
\end{align}
Recall that given $x(i)$ the jump wait time $\Delta \tau_i$ has a $\text{Geom}[x(i)/2]/(4n^{2})$ distribution.
Using this it is elementary to show that $\tau_I-t\leq Cn^{-1.99}$ almost surely, with the constant depending only on $B,W,\beta$.
Since $X_n(\tau_I)=X(\wt{\tau}_I)$ and $|\tau_I-\wt{\tau}_I|\leq Ctn^{-0.499}$ we have $\wt{\tau}_I\leq 1.01t$ for all sufficiently large $n$.
Therefore the hypothesis $\ul{X}(1.01t)\geq 1/\log n$ implies $\ul{x}_I\geq 1/\log n$.
Using this we can bound the first and third terms on the right hand side of (\ref{pfEq938}) within the tolerance of the lemma statement.
Using lemma \ref{HMartProp01}, we bound the second term.
\end{proof}


\subsection{Estimating $|\Phi^{A}-\wt{\Phi}^{A}_n  |$}

Throughout this section we use the notation 
\begin{align}
I(x)=\int_{x}^1 \frac{2\,dW(y)}{\sqrt{\beta y}}.
\end{align}

\begin{lemma} \label{TBound}
On the event $\ul{X}(t)\geq 1/\log n$ we have the following bounds for almost all $W$:
\begin{align}
\sup_{1/\log n\leq x\leq 1}| I(x) | \leq & C  (\log \log n)^{0.51}, \label{IBound}\\
\sup_{1/\log n\leq x\leq 1}|s'(x)| \leq & C \beta^{-1/2} (\log n)^{1.01}, \\ 
T^{-1}(t) \leq &  C\beta^{-1/2} t (\log n)^{1.02}.
\end{align}
The constant $C$ depends on $W,\beta$.
\end{lemma}
\begin{proof}
For another Wiener process $\wt{W}$ we have
\begin{align}
I(x) =\frac{2}{\sqrt{\beta}} \wt{W}_{-\log x}.
\end{align}
The bound (\ref{IBound}) now follows from the law of the iterated logarithm.
Recall that
\begin{align}
s'(x) = &x^{-1} e^{I(x)} ,\qquad 
T^{-1}(t) = \int_{0}^t 2 X_{u}^{-1}e^{2 I(X_u)}  \, du.
\end{align}
The other two bounds now follow from the above display and $\ul{X}(t)\geq 1/\log n$.
\end{proof}
We will use the following bound on the supremum of Brownian local time.
\begin{lemma} \label{LocTimeBound}
For almost all $B$ we have
\begin{align}
\sup_{x\in \mathbb{R},t>0} \frac{L_B(x,t)}{\sqrt{t \log t}\vee 1} <\infty .
\end{align} 
\end{lemma}
\begin{proof}
The key ingredient is inequality (4.6) of A.\ N.\ Borodin's survey \cite{BorodinSurvey}, which is
\begin{align}
P\left(\sup_{x\in \mathbb{R}} L_B(x,t)>r\right) \leq  \frac{A r^2}{t} \exp\left(\frac{-r^2}{2t}\right) \label{BorodinIneq}
\end{align}
for some constant $A$.
We start with a union bound
\begin{align}
P\left(\sup_{x\in \mathbb{R},t>0} \frac{L_B(x,t)}{\sqrt{t\log t}\vee 1} > R \right)
\leq &\sum_{k=0}^{\infty} P\left(\sup_{x\in \mathbb{R}} \frac{L_B(x,k +1)}{\sqrt{\lfloor k \log k\rfloor }\vee 1} >  R \right).
\end{align}
Applying (\ref{BorodinIneq}), we see that the sum on the right hand side converges, and the result is a rapidly decaying function of $R$.  Considering only positive integer values of $R$, our assertion follows from the Borel-Cantelli lemma.
\end{proof}

\begin{lemma}
On the event $\ul{X}(t)\geq 1/\log n$, for all $B,W$ we have
\begin{align}
|\Phi^{A}-\wt{\Phi}^{A}_n  |\leq  C tn^{-0.99}.
\end{align}
The constant depends on $B,W,\beta$.
\end{lemma}
\begin{proof}
From the definitions it is easy to see that
\begin{align}
\Phi^{A}-\wt{\Phi}^{A}_n = \sum_{k=1}^n A_k ,\qquad A_k=\int_{x_{k-1}}^{x_{k}} \left(\frac{1}{x_k} -\frac{1}{x} \right) L_X(x,t) \, dx.
\end{align}
Using formula (\ref{changeLocTime}) to express local time for $X$ in terms of local time for $B$ gives
\begin{align}
 A_k =  \int_{x_{k-1}}^{x_k}  \left(\frac{1}{x_{k-1}}-\frac{1}{x}\right) \frac{L_B(s(x),T^{-1}(t))}{2\exp(I(x))}\, dx \label{fpmdt}
\end{align}
For $k\geq \frac{n}{\log n}$ we have $|x_{k-1}^{-1}-x^{-1}|\leq C\log (n)/n$.
The lemma statement now follows immediately by using lemmas \ref{TBound} and \ref{LocTimeBound} to estimate the right hand side of (\ref{fpmdt}).
\end{proof}


\section{Convergence of noisy terms, part I} \label{section5}

Recall that in display (\ref{MakePhiTilde}) we introduced an intermediate functional $\wt{\Phi}_n$, and decomposed the functionals $\Phi,\wt{\Phi}_n,\Phi_n$ into ``noiseless" parts $\Phi^A$, etc., and ``noisy" parts $\Phi^B$, etc.
It this section we estimate the difference between the terms
\begin{align}
\Phi^{B}_n(t)=&
\frac{-1}{4n^2} \sum_{i=1}^{\lfloor 4tn^2 \rfloor} 
 \frac{ \sqrt{n}G_{n x(i)}}{\sqrt{\beta x(i)}} +\frac{   G^{(2)}_{nx(i)} }{\beta x(i)}\\
 \wt{\Phi}^{B}_n (t) =&  -\sum_{k=1}^n  \left(\frac{  \sqrt{n} (G_{k+1}+G_k)}{2\sqrt{\beta x_k}} +\frac{   G^{(2)}_k}{ \beta x_k}   \right) L_{X}([x_k,x_{k+1}) ,t).
 \end{align}
The results of this section build up to a proof of the following.
\begin{proposition} \label{PhiNPhiTildeThm}
On the event $\ul{X}(t)\geq 1/\log n$, for almost all $B,W$ we have
\begin{align}
 \left| \Phi^{B}_n(t)-\wt{\Phi}^{B}_n(t)\right| \leq C tn^{-0.249} .
\end{align}
The constant depends on $B,W,\beta$.
\end{proposition}

\subsection{Mean increments for $\Phi^{B}_n - \wt{\Phi}^{B}_n$}
 \label{hMartSection}

In this section we will again use the notations $\tau_i,\wt{\tau}_i,\Delta \wt{\tau}^{\uparrow}_i$, etc.\ defined in section \ref{SectionAAN}.
We will however reassign the notation $\xi_i$ to the increments that we study here, that is
\begin{align}
\Phi^{B}_n(\tau_{m}) -  \wt{\Phi}^{B}_n( \tau _{m}) =& \sum_{i=1}^m \xi_i \\
\xi_i
=&
 \frac{  G^{(2)}_{k}}{ \beta x_k}  \left( \Delta \tau_i - \Delta \wt{\tau}_i \right)  
 + \frac{ G^{(2)}_k}{\beta} \left(\frac{1}{x_{k-1}}-\frac{1}{x_k}\right)\Delta \wt{\tau}^{\downarrow} \label{xiDef2} \\
&\qquad 
+ \frac{  \sqrt{n}(G_k+G_{k-1})}{2\sqrt{\beta}} \left( \frac{1}{\sqrt{x_{k-1}}}-\frac{1}{\sqrt{x_k}}\right)
\Delta \wt{\tau}^{\downarrow} \nonumber \\
&\qquad +\frac{ \sqrt{n}}{2\sqrt{\beta x_k}}\left(2G_{k }\Delta \tau_i -(G_{k +1}+G_{k })\Delta \wt{\tau}^{\uparrow}_i     -(G_{k }+  G_{k -1} ) \Delta \wt{\tau}^{\downarrow}_i   \right).\nonumber
\end{align}
Let us recall the argument we used in section \ref{NoNoiseSection} to analyze the noiseless ``$A$" terms.
In that section the increments we considered had conditional means of order $O(n^{-5/2})$ (forgetting the powers of $\log n$) and variances of order $O(n^{-4})$.
Since we had a sum of $O(n^2)$ terms, a central limit theorem heuristic suggests a result with mean of order $O(n^{-1/2})$ and fluctuations of $O(n^{-1})$.
Here this heuristic no longer works because, as the next lemma shows, the conditional means of the $\xi_i$'s are on about order $O(n^{-3/2})$:
\begin{lemma} \label{meanCor}
For almost all $W$, there exists a constant $C=C(W ,\beta )$ such that for all $n\in \mathbb{N}$ and all $k=\frac{n}{\log n},\ldots ,n$, the approximation 
\begin{align}
|\E\left[ \xi_i \middle| W,x(i)=x_k\right] -\mathcal{E}_k  |\leq C  n^{-5/2}\log( n)^{9/2} 
\end{align} 
holds, where
\begin{align} \label{EDef}
\mathcal{E}_k
&=
\frac{  \sqrt{n} \Delta t}{2\sqrt{\beta }x_{k}^{3/2}}\bigg\{  
  G_{k+1}-2G_k +G_{k-1} 
+
\frac{G_{k+1}G_k -G_k G_{k-1}}{ \sqrt{\beta k}} \\
& \hspace{80pt}+
\frac{1}{ \sqrt{\beta k}}\left((G_{k+1}+G_k)\gamma_k -(G_k +G_{k-1})\gamma_{k-1}\right)
\bigg\} . \nonumber
\end{align}
\end{lemma}
\begin{proof}
Substitute the formulas from lemma \ref{exitTimeLemma01} in formula (\ref{xiDef2}).
\end{proof}
Therefore at first glance the situation appears desperate-- there is an entire extra power of $n$ we must overcome!
However, the argument from section \ref{SectionAAN} at least reduces our task from analyzing a sum of $\xi_i$'s to analyzing the following sum:
\begin{align}
\mathcal{S}_i =\sum_{j=1}^i \mathcal{E}_{n x(j)} \label{SDef}.
\end{align}
Indeed we find the following analogue of lemma \ref{HMartProp01}.
\begin{lemma} \label{HMartProp}
On the event $\ul{X}(1.01t)\geq 1/\log n$ we have, for all $n\in \mathbb{N}$ that
\begin{align}
\left|\Phi_n(\tau_i) -\wt{\Phi}_n(\wt{\tau}_i)- \mathcal{S}_i\right| \geq C t n^{-0.499}.
\end{align}
The above holds for almost all $B,W$ and $C=C(B,W,\beta)$.
\end{lemma}
\begin{proof}
The proof is the same as that of lemma \ref{HMartProp01}, so we omit the details.
\end{proof}

\subsection{Analysis of $\mathcal{S}_i$}  \label{BranchingSection}

In the previous section we reduced the comparison of $\Phi^{B}_n$ and $\wt{\Phi}^{B}_n$ to the estimation of the sum $\mathcal{S}_i$ defined in formulas (\ref{EDef},\ref{SDef}).
Let $D_k(t)$ be the number of downward jumps by $X^n$ from height $x_k$ up to time $t$ and $\ol{D}(t)=\max_{k=1}^n D_k(t)$. 
In this section we prove the following proposition.
\begin{proposition}  \label{VMartProp}
For almost all $W$, all $n\in\mathbb{N}$, and all positive integers $i$, we have
\begin{align}
P\left(|\mathcal{S}_i | \geq   n^{-0.249}, \ul{x}(i)\geq \frac{1}{\log n},\ol{D}(\tau_i)\leq  n (\log n)^3\right)
\leq & \exp\left(-cn^{0.5}\right),
\end{align}
where $c=c(W,\beta)>0$.
\end{proposition}
Clearly this proposition is only useful together with a bound on the probability that $\ol{D}(\tau_i)$ is large; that will be the subject of section \ref{NStarSection}.  

\subsubsection{Branching structure for numbers of jumps}
Let $J_k(t)$ be the total number of jumps by $X^n$ from height $x_k$ up to time $t$.  
Then letting $U_k(t)$ and $D_k(t)$ respectively be the numbers of upward and downward jumps from height $x_k$, we clearly have $J_k(t)=U_k(t)+D_k(t)$ and
\begin{align} \label{EFormula2}
\mathcal{S}_i=\sum_{j=1}^i \mathcal{E}_{nx(j)} =\sum_{k=1}^n J_k\left(\tau_i\right) \mathcal{E}_k.\qquad
\end{align}
Among the first $i$ jumps, by definition $J_k(\tau_i)$ of them are from $x_k$.
It is somewhat redundant to keep track of both upward and downward jump counts since it is easy to see that
\begin{align} \label{UFormula1}
U_k(t) =&D_{k+1}(t)+\begin{cases}  1 &\text{ if } X^n(0) \leq x_k< X^n(t) \\
-1 & \text{ if } X^n(0)>x_k\geq  X^n(t) \\ 0 &\text{ else.}
\end{cases}
\end{align}
A key ingredient in our analysis will be a branching structure for the numbers of upward and downward jumps.
However the branching structure can only be applied at appropriate stopping times.
Let $T_{i,k'}$ be the time of the $i^{th}$ visit by $X^n$ to height $x_{k'}$.

To state the branching structure, we define the following notation:
\begin{align}
d_k = \begin{cases} 1 & \text{if }X^n(0) \leq x_k \leq x_{k'} \text{ or } X^{n}(0)\geq x_k \geq x_{k'} \\
0 & \text{else.}\end{cases}
\end{align}
The branching structure is the following lemma:
\begin{lemma} \label{BranchingLemma} 
Fix positive integers $i,k'$ and let $D_k=D_k(T_{i,k'})$, $U_k=U_k(T_{i,k'})$.
There exist independent geometric random variables $\{\gamma_{\ell,k}:\; 1\leq k\leq n,\ell\geq 1\}$ such that:
for $k=1,\ldots ,k'-1$ we have
\begin{align} \label{branching1}
D_{k} =& \sum_{\ell=1}^{D_{k+1}+d_k}  \gamma_{k,\ell} ,\qquad
\gamma_{k,\ell} +1 \eqd \mathrm{Geom}[p_{k}],
\end{align}
and for $k=k'+1,\ldots ,n-1$ we have 
\begin{align} \label{branching2}
U_{k} =& \sum_{\ell=1}^{U_{k-1}+d_k}  \gamma_{k,\ell}   ,\qquad
\gamma_{k,\ell}+1 \eqd \mathrm{Geom}[q_{k}]. 
\end{align}
\end{lemma}	
\begin{proof}
First assume that $x_k<x_k'$ and $X^n(0)<x_{k}$.
Let $\gamma_{k,1}$ be the number of downward steps from $x_k$ before the first upward step from height $x_{k}$, and for $\ell\geq 1$ let $\gamma_{k,\ell}$ be the number of downward steps from $x_k$ between the $(\ell-1)^{th}$ and $\ell^{th}$ upward steps from $x_k$.
Clearly the variables $\{\gamma_{k,\ell}:\;\ell\geq 1\}$ are independent.
Also clearly each $\gamma_{k,\ell}+1$ has a geometric distribution with parameter equal to the probability that the next jump from $x_k$ is upwards, which is $p_k$.

Since $X^n(0)\leq x_k$, the number of downward steps from $x_k$ before the first downward step from $x_{k+1}$ is $\gamma_{k,1}$ (in the other case this would be zero).
Furthermore, the number of downward steps from $x_k$ before the $\ell^{th}$ downward step from $x_{k+1}$ is
\begin{align}
\sum_{i=1}^{ \ell} \gamma_{k,i}.
\end{align}
Since $D_{k+1}=D_{k+1}(T_{i,k'})$ and since $k'>k$, there cannot be any downward steps from $x_k$ between $T_{i,k'}$ and the $(D_{k+1}+1)^{th}$ downward step from $x_{k+1}$.
Thus the number of downward steps from $x_k$ before $T_{i,k'}$ equals the number before the $(D_{k+1}+1)^{th}$ downward step from $x_{k+1}$.
This establishes the formula for $D_{k}$ in the lemma statement for the case we are considering.
%
%
%
%
%
%

The proofs of the other possible cases of the lemma are similar.   
We only comment that if $k>k'$, then $\gamma_{k,\ell}$ is defined instead to be the number of upward steps from $x_k$ between the $(\ell-1)^{th}$ and $\ell^{th}$ downward steps from $x_k$.
\end{proof}

\subsubsection{Apply summation by parts to $\mathcal{S}_i$}

By lemma \ref{BranchingLemma} and since a geometric random variable with parameter $p$ has mean $1/p$, 
we have
\begin{align}
\E \left[ D_{k}-D_{k+1} \middle| D_k,W\right]
=& \left( \frac{q_{k}}{p_{k}}-1\right)D_k
=\left(\frac{2G_k}{\sqrt{\beta k}} +O( \log(n)^3 n^{-1}) \right)D_k.
\end{align}
For the above approximation we used lemma \ref{pkApproxLemma}.
The above calculation suggests that differences of the $D_k$ variables are relatively small.
Therefore it is useful to write $\mathcal{S}_i$ in terms of the differences $D_k -D_{k+1}$.
In the next lemma we achieve this using summation by parts.
There is one subtle detail though: we do not use summation by parts on the term $D_{k+1}f_{2,k}$ in the lemma statement; this is because we need this term ``left behind" so that it can be part of a cancellation in the proof of lemma \ref{EQMGFLemma}.

In the following lemma we use the notation $D_\ast$ for a hypothetical bound on the maximum number of downward steps from any height; a reasonable value to use would be $D_\ast =n\log n$.
\begin{lemma} \label{sumByPartsLemma}
Let $D_k=D_k(\tau_i)$.
For almost all $W$, on the event $\{\ol{D}(\tau_i)\leq D_\ast , \ul{x}(i)\geq \frac{1}{\log n}\}$ 
 we have $|\mathcal{S}_i-S_i|\leq C (\log n)^4D_\ast n^{-3/2}$ where
\begin{align} \label{fDef}
S_i 
=&
\frac{1}{n^2}\sum_{k=1}^n   \left\{ (D_{k}-D_{k+1 })(\sqrt{n}f_{1,k} +f_{3,k})+  D_{k+1}f_{2,k} \right\} 
\\
f_{1,k} =&\frac{  G_{k+1}-G_{k-1} }{8\sqrt{\beta} x_{k}^{3/2}},\qquad
f_{2,k}=\frac{1}{4\beta x_{k}^2}G_{k}(G_{k+1}-G_{k-1}) \\
f_{3,k}=&\frac{1}{8\beta x_{k}^2} (G_{k+1}G_k -G_{k}G_{k-1}+\gamma_k(G_{k+1}-G_{k})+\gamma_{k-1}(G_{k}-G_{k-1})).
\end{align}
\end{lemma}

\begin{proof}
By formulas (\ref{EFormula2},\ref{UFormula1}) there exist numbers $\wt{d}_k \in\{-1,0,1\}$ such that
\begin{align}
\mathcal{S}_i
=&
\sum_{k=1}^n (D_k +D_{k+1}+\wt{d}_k ) \mathcal{E}_k. \label{aaeoringa}
\end{align}
By the bound $|G_k|\leq \log n$ and the condition $\ul{x}(i) \geq 1/\log n$, we have $|\mathcal{E}_k| \leq C n^{-3/2} (\log n)^4$ almost surely.
Therefore the term $\wt{d}_k \mathcal{E}_k$ in equation (\ref{aaeoringa}) can be ignored.

To explain precisely how we apply summation by parts, decompose $\mathcal{E}_k$ as follows:
\begin{align} 
\mathcal{E}_k =&c_k (H_{k+1}-H_{k})+c'_k ( H'_{k+1}-H'_k) +c''_k (H''_{k+1} - H''_{k}), \qquad\text{where} \label{kajbr}\\
c_{k} =&\frac{1 }{8\sqrt{\beta}k^{3/2}},\qquad c'_{k}=c''_k =\frac{1 }{8\beta k^2}, \\
H_{k}=& G_{k}-G_{k-1},\qquad H'_k = G_{k}G_{k-1},\qquad 
H''_{k}=(G_k+G_{k-1})\gamma_{k-1}.
\end{align}
We will then use
\begin{align}\label{wrugb}
&\sum_{k=1}^n (D_k +D_{k+1})c_k (H_{k+1}-H_k)\\
&= 2D_{n+1}c_{n+1}H_{n+1}-2D_{1}c_1 H_1
 -\sum_{k=1}^n (D_{k+1}-D_k)c_k (H_{k+1}+H_k)+2D_{k+1}(c_{k+1}-c_k)H_{k+1}\\ 
&=  \bigg(-\sum_{k=1}^n (D_{k+1}-D_k)c_k (H_{k+1}+H_k) \bigg) +O(D_\ast \log(n)^{7/2} n^{-3/2}). 
\end{align}
To ignore the second term inside the sum we used that $(k+1)^{-3/2}-k^{-3/2} =O(k^{ -5/2})$ and that since $\ul{x}(i)\geq 1/\log n$ we have $D_k=0$ for all $k\leq n/\log(n)$.   

We treat the $c''_k(H''_{k+1}-H''_k)$ term on line (\ref{kajbr}) in the same way as calculation (\ref{wrugb}).

For the $c'_k (H'_{k+1}-H'_{k})$ term on line (\ref{kajbr}), we do not sum by parts.  
Instead we make the following algebraic manipulation:
\begin{align}
&\sum_{k=1}^{n} (D_{k} + D_{k+1})c'_k (H'_{k+1}-H'_{k}) \\
&=\sum_{k=1}^{n}2 D_{k+1} c'_k (H'_{k+1}-H'_{k}) +(D_{k}-D_{k+1})c'_k (H'_{k+1}-H'_k).
\end{align}
It follows that the term $f_2$ in the lemma statement is given by $f_2=2n^2 c'_k (H'_{k+1}-H'_k)$.
Collecting all the remaining terms gives the formula asserted in the lemma statement.
\end{proof}

\subsubsection{MGF bound for the remainder term}

In the previous section we showed that instead of $\mathcal{S}_i$, it suffices to analyze terms $S_i$ that were obtained after applying summation by parts.
We now write the $S_i$ terms as follows:
\begin{align}
S_i=\sum_{k=1}^n E_{k,i},\qquad E_{k,i}=\frac{1}{n^2}(D_k(\tau_i) -D_{k+1}(\tau_i))(\sqrt{n} f_{1,k}+f_{3,k}) +\frac{1}{n^2}D_{k+1}(\tau_i)f_{2,k},
\end{align}
Recall that the terms $f_{i,k}$ were defined on line (\ref{fDef}) as functionals of $W$; they are local to the region around $x_k$ and typically of order $O(1)$.
Throughout this section we will abbreviate $D_k=D_k(\tau_i)$ when it is clear at what time we are counting the downward steps.
As we mentioned above, the purpose of the summation by parts was to take advantage of the differences $D_{k}-D_{k+1}$, which we now do:
\begin{lemma}\label{EQMGFLemma}
Suppose $j$ is defined by $\tau_j =T_{i,k'}$ for some $i,k'$.
For almost all $W$, all $n\in \mathbb{N}$, all $k=\frac{n}{\log n},\ldots ,n$ and $|\lambda| \leq n$ the following approximation holds:
\begin{align}
\log \mathbb{E}\left[ \exp\left(\lambda E_{k,j}\right)\middle| D_{k+1},W\right] =D_{k+1}\left(\frac{f_{1,k}^2}{n^3}\lambda^2 +O\left((\log n)^p n^{-3/2}\right)\right),
\end{align}
where the implied constant depends only on $W,\beta$.   Here $p$ is a positive universal constant.
\end{lemma}
\begin{proof}
We will assume that $k\leq k'$.  
A similar proof based on formula (\ref{branching2}) can be given for the case $k\geq k'$.
Given $D_{k+1}$ and $W$, the variable $D_{k}$ is (by lemma \ref{BranchingLemma}) a sum of $N_{k+1}$ independent $\text{Geom}[p_k]-1$ random variables.
It follows that
\begin{align}
\mathbb{E}\left[\exp(\lambda D_k\middle| D_{k+1}=D,W\right]
=&
\left(1 -(q_{k}/p_{k})(e^\lambda -1)\right)^{-D}.
\end{align}
The lemma now follows by substituting the definition of $E_{k,j}$ in this explicit formula for the moment generating function, and making a Taylor approximation:
\begin{align}
\log \mathbb{E}\left[ \exp\left(\lambda E_{k,j}\right)\middle| D_{k+1},W\right]
=&
\frac{D_{k+1}\lambda}{n^2}(\sqrt{n}f_{1,k}+f_{2,k}+f_{3,k})\\
&-D_{k+1}\log\left(1-\frac{q_k}{p_k}\left(\exp
\left(-(\sqrt{n}f_{1,k}+f_{3,k})
\lambda/n^2\right)-1\right)\right).
\end{align}
Lemma \ref{pkApproxLemma} gives the series expansion $q_{k}/p_k = 1-2G_k/\sqrt{\beta k}+O(\log(n)^3/n)$. The lemma statement now follows by using Taylor series for the logarithm and exponential.
\end{proof}

%

\subsubsection{Proof of proposition \ref{VMartProp}}
\begin{proof}[Proof of proposition \ref{VMartProp}]
By lemma \ref{sumByPartsLemma} it suffices to replace $\mathcal{S}_i$ with $S_i$ in the proposition statement.

For some $i',k'$ we have either $\tau_i=T_{i',k'}$.
Although $k',i'$ are random, by conditioning on their values it suffices to obtain a bound that holds uniformly in $k',i'$.
Within this proof we abreviate $D_k=D_k(T_{i',k'})$.
We will also use the notation 
\begin{align}
\ol{D}_k =\max\{ D_j :\; j=k,\ldots n\}.
\end{align}
Suppose $k\leq k'$ and let $Z_k(\lambda)=\exp (\lambda 
\sum_{j=k}^{k'}E_{j,i})$.
Also for each $j$ let $F_j$ be the $\sigma$-field generated by $D_{j},\ldots ,D_n$ and $W$.  
In the following calculation we let $D_\ast=n(\log n)^3$.
For the second inequality we use lemma \ref{EQMGFLemma}.
Also we assume $k\geq  n/\log n$.
\begin{align}
\E\left[ Z_k(\lambda) \one_{\ol{D}_k\leq D_\ast}\middle|W \right]
\leq &
\E\left[ Z_{k+1}(\lambda) \one_{\ol{D}_{k+1}\leq D_\ast}   \E \left[ \exp\left(\lambda E_{k,i} \right) \middle| F_{k+1}\right]\middle|W \right] \nonumber \\
\leq &
\E\left[ Z_{k+1}(\lambda) \one_{\ol{D}_{k+1}\leq D_\ast} 
\exp\left(D_{k+1}f_{1,k}^2 n^{-3}\lambda^2 +CD_{k+1}\log(n)^p n^{-3/2} \right)\middle|W \right]\nonumber \\
\leq &
\E\left[ Z_{k+1}(\lambda) \one_{\ol{D}_{k+1}\leq D_\ast}\middle|W  \right]
\exp\left(C D_{\ast}\log(n)^p \left(n^{-3}\lambda^2 +  n^{-3/2}\right)\right).
\end{align}
It follows by induction that
\begin{align} \label{389429}
\E\left[ Z_{\lfloor n/\log n \rfloor}(\lambda) \one_{\ol{D} \leq D_\ast}\middle| W\right]
\leq &
\exp\left(C D_{\ast}\log(n)^p \left(n^{-2}\lambda^2 +  n^{-1/2}\right)\right).
\end{align}
Since we are working on the event $\ul{x}(i)\geq 1/\log n$, we have that $Z_1 =Z_{\lfloor n/\log n\rfloor }$.

By Chebychev's inequality and line (\ref{389429}) we get
\begin{align}
&\log P\bigg(\sum_{j=1}^{k'} E_{j,i} \geq r,\ol{D}(\tau_i) \leq D_\ast,\ul{x}(i)\geq \frac{1}{\log(n)}\bigg|W\bigg)\\
&\qquad \leq
\log P\bigg( \sum_{j=n/\log n}^{k'} E_{j,i}\geq r,\ol{D}(\tau_i)\leq D_\ast \bigg|W\bigg)\\
&\qquad \leq 
 -\lambda r +C D_{\ast}\log(n)^p \left(n^{-2}\lambda^2 +  n^{-1/2}\right).
\end{align}
Now obtain a bound by letting $r=(D_\ast/n)n^{-0.2499}$ and $\lambda=n^{0.75}$.
The same bound can be obtained for the probability that $\sum_{j=k'}^n E_{j,i}  \geq n^{-0.249}$ by a similar argument.
\end{proof}


\subsection{Bound on numbers of jumps} \label{NStarSection}

In this section we will use the notations 
\begin{align}
\ol{D}_k(t)=\max_{i=k}^n D_i(t),\qquad 
\ul{X}(t)=\inf \{X(s):\; 0\leq s\leq t\}.
\end{align}
After several lemmas we will prove the following
\begin{lemma} \label{NStarLemma}
Let $D_\ast \geq   t n (\log n)^3$.
For almost all $W$ there exists $c=c(W,\beta)>0$ such that
\begin{align}
P\left(\ol{D}(t) \geq D_\ast ,\ul{X}(t)\geq \log(n)^{-1}  \middle| W  \right)\leq \exp\left( -\frac{c  D_{\ast}^2}{ t n^2} \right).
\end{align}
\end{lemma}
We will divide the proof of lemma \ref{NStarLemma} into three pieces, two of them being lemmas \ref{mgflemma4} and \ref{DeltaNLemma}.
Before getting into the details, we would like to explain roughly the idea of the proof.
By a union bound it suffices to estimate the probability of an unusually large number of downward steps at some height $x_k\geq 1/\log n$.
Because of the branching structure from lemma \ref{BranchingLemma}, if there are many downward steps from $x_k$ there will also probably be many downward steps from nearby heights.
To make this precise use the branching structure to compute a conditional moment generating function (lemma \ref{mgflemma4}) and then apply Chebychev's inequality (lemma \ref{DeltaNLemma}).
There can be at most $4tn^2$ steps by time $t$, so if there are sufficiently many downward steps from nearby heights that the total number of steps exceeds $4tn^2$, then we have a contradiction.
In some sense we are reducing the problem of bounding $\ol{D}(t)$ to the task of proving a high probability modulus of continuity estimate on $D_k(t)$ considered as a function of $k$.

In this section we choose and fix a starting height $k_0/n$ for the process $X^n$, and also integers $i$ and $k'$ which index a stopping time $T_{i,k'}$.
In this section the time variable will always be this $T_{i,k'}$, so for example $D_{k}=D_{k}(T_{i,k'})$.


\begin{lemma} \label{mgflemma4}
If $k<k'$ then the conditional moment generating function for $D_k$ is
\begin{align}
\mathbb{E}[e^{\lambda D_k} | W, D_{k'}=D]
=& \left( 1+ \frac{ (e^\lambda-1)\Delta S}{1-  (e^\lambda -1)S} \right)^{D}
\end{align}
where $\Delta S=\prod_{j=k}^{k'-1} q_j/p_j $ and $S=\sum_{i=k}^{k'-1} \prod_{j=k}^{i} q_j/p_j $.
\end{lemma}
\begin{proof}
Let $F_j$ and $f_j$ be the following probability generating functions:
\begin{align}
F_j(z)=\sum_{j=0}^\infty P(D_j=\ell | W,D_{k'})z^\ell,\qquad
f_j(z)=\sum_{j=0}^\infty P(D_j=\ell | W,D_{j+1})z^\ell.
\end{align}
If $j<k'$ then lemma \ref{BranchingLemma} asserts that $D_j$ is a sum of independent variables $\gamma_{j,\ell}$ such that $\gamma_{j,\ell}+1$ has a geometric distribution.
A short calculation using the formula for the probability generating function for a geometric random variable shows that
\begin{align}
f_j(z) =& \left(1 -\rho_j (z-1)\right)^{-1},\qquad \rho_j =\frac{q_j}{p_j}.
\end{align}
By the general theory of braching processes, $F_j (z) =F_{j+1}\circ f_j(z)$.
Suppose that $D_{k'}=1$, so that $F_{k'}(z)=z$.
It is easy then to check by induction for $k<k'$ we have
\begin{align}
F_k(z)  =&  1+\frac{(z-1)\prod_{j=k}^{k '-1}\rho_j}{1-(z-1)\sum_{i=k}^{k'-1}\prod_{j=k}^{i} \rho_j}.
\end{align}
The formula in the lemma statement now follows since the moment generating function and probability generating function are related by $M(\lambda)=F(e^\lambda)$.
The general case $D_{k'}\geq 1$ follows from the composition formula for probability generating functions.
\end{proof}


\begin{lemma} \label{DeltaNLemma}
Suppose that $k <k'$, and suppose the random environment $W$ is such that
\begin{align}
\Delta S=& \prod_{j=k}^{k'-1}\frac{p_j}{q_j} \in [0.99,1.01],\qquad \text{and}\qquad
S=\sum_{i=k}^{k'-1}\prod_{j=k}^i \frac{p_j}{q_j}\geq 100.
\end{align}
Then the following bound holds:
\begin{align}
P\left(D_{k}< D_{k'}/2 \middle|W, D_{k'} \right)\leq & \exp\left(-0.05  D_{k'}/S\right).
\end{align}
\end{lemma}
\begin{proof}
If $X$ is a random variable with a finite moment generating function and $\lambda>0$, then Chebychev's inequality gives the following bound on lower tail probabilities:
\begin{align}
\log P(X<r)\leq & \lambda r +\log \mathbb{E}[\exp(-\lambda X)].
\end{align}
Starting with the result of the previous lemma, we let 
\begin{align}
\lambda=-\log(1+\alpha),\qquad \text{where }\alpha=\frac{-1}{4S}.
\end{align}
We then apply Taylor's theorem to the logarithm function to get
\begin{align}
&\log P(D_{k } <D/2 |W,D_{k'}=D ) 
\leq  \lambda D/2 + D  \log  \left( 1+ \frac{ (e^{-\lambda}-1)\Delta S}{1-  (e^{-\lambda} -1)S} \right) \\
&=
- (D/2) \log (1+\alpha) +D  \log \left(1-\alpha(S- \Delta S)\right)-D\log\left(1-\alpha S\right) \label{eq114} \\
&=(D \Delta S  -D/2 ) \alpha +\left(D/2   + 2 D S \Delta S - D (\Delta S)^2\right)\frac{\alpha^2}{2} +\frac{\alpha^3}{6}   f'''(\xi) ,  \label{eq115}
\end{align}
where $\xi=\xi(\alpha) \in [\alpha,0]$.
Here $f$ is the right hand side of line (\ref{eq114}) considered as a function of $\alpha$.
Since $\alpha^3 \leq 0$ we now want to get a lower bound for $f'''(\xi)$.
Assume that $r>0$, and note that $-1<\xi<0$ and $S>\Delta S$. It follows that
\begin{align}
f'''(\xi)=& -\frac{D}{(1+\xi)^3} - \frac{2D(S-\Delta S)^3}{(1-\xi (S-\Delta S))^3} +\frac{2DS^3}{(1-\xi S)^3} 
\geq  -\frac{D}{(1+\xi)^3}.
\end{align}
The above inequality is valid because $x\mapsto x/(1-\xi x)$ is an increasing function.

Now we use the hypotheses that $\Delta S\in [0.99,1.01]$ and $S\geq 100$.
These conditions imply the following bounds for the three terms on line (\ref{eq115}):
\begin{align}
(D\Delta S -D/2)\alpha \leq & -0.1225 D/S,\\
\left(D/2 + 2 D S \Delta S - D (\Delta S)^2\right)\frac{\alpha^2}{2}  \leq & 0.0630 D/S, \\
\frac{\alpha^3}{6} f'''(\xi) \leq &   3\times 10^{-7} D/S. 
\end{align}
Our assertion follows by adding these inequalities.
\end{proof}


\begin{lemma} \label{DeltaSLemma}
Almost all realizations of the random environment $W$ satisfy the following conditions. 
There exists $c=c(\beta,W)>0$ such that
for sufficiently large $n$ (depending on $\beta,W$), and all $k,k'$ such that $n/\log n\leq k<k' \leq n$ and $k'-k\leq c n/( \log n)^2$, we have
\begin{align}
\prod_{j=k}^{k'-1} \frac{p_j}{q_j} \in \left[ 0.99,1.01\right].
\end{align}
\end{lemma}
\begin{proof}
By lemma \ref{pkApproxLemma} we have (for almost all $W$)
\begin{align}
\log \prod_{j=k}^{k'-1} \frac{p_j}{q_j} =\sum_{j=k}^{k'-1} \left\{ \frac{-2G_j}{\sqrt{\beta j}}+\frac{1}{j} +\frac{2G_{j}^{(2)}}{\beta j} +O( ( \log  n)^3 n^{-3/2})\right\}. \label{eq937}
\end{align}
We will show that the right hand side is $O(1/\log n)$ for almost all $W$.
Since $k'-k\leq n/ \log n $ we see that ``$O$'' terms can be ignored.
Recall that the variables $G_k$ are defined in equation (\ref{GDef}) as certain integrals of the random environment.
It follows that
\begin{align}
\sum_{j=k}^{k'-1} \frac{G_j}{\sqrt{j}} 
=&  \sum_{j=k}^{k'-1}  \int_{x_{j-1}}^{x_{j+1}} \frac{1}{\sqrt{x_j}} (1-n|x-x_j|)\, dW(x). \label{eq938}
\end{align}
The right hand side is a Gaussian random variable and it is straitforward to bound its variance by $C \log (n)(k'-k)/n \leq Cc/   \log n  $.
By a standard tail bound for the normal distribution, the probability that the left hand side of (\ref{eq938}) exceeds $0.01$ is bounded by $ \exp(-  (0.01)^2  \log (n)/(Cc))$.
Taking $c$ sufficiently small and making a union bound over $k,k'$, the Borel-Cantelli lemma now shows that the left hand side of (\ref{eq937}) is less than $0.01$ for almost all $W$.

The $G^{(2)}_j/j$ terms in equation (\ref{eq937}) can be analyzed in a similar way but are smaller by a factor of $\sqrt{j}$.
The $1/j$ terms make a contribution of $O(1/\log n)$ and therefore can be ignored.
The lemma statement then follows by taking the exponential of both sides and using a Taylor approximation.
\end{proof}


\begin{proof}[Proof of lemma \ref{NStarLemma}]
In this proof we fix a realization of $W$ such that the conclusion of lemma \ref{DeltaSLemma} holds, and all probabilities are conditional on this $W$.

In order to use lemma \ref{DeltaNLemma} we must first we reduce to the case that $t$ is one of the stopping times  $T_{i,k'}$.
The number of downward steps from any height can only change at one of the stopping times $T_{i,k'}$.
Thus by a union bound we have
\begin{align}
& P\left(\ol{D}(t) \geq D_\ast,\ul{X}(t)\geq 1/\log n\right) \\
& \leq 
\sum_{i=1}^{\lfloor 4tn^2\rfloor} \sum_{k'= \lfloor n/\log n \rfloor}^{n} P\left(\ol{D}(T_{i,k'})\geq D_\ast, \ul{X}(t)\geq 1/\log n,T_{i,k'}\leq t\right).
\end{align}
The condition $X(t)\geq 1/\log n$ let us assume $k'\geq n/\log n$.  By another union bound we get
\begin{align}
P\left(\ol{D}(T_{i,k'})\geq D_\ast ,\ul{X}(t)\geq 1/\log n,T_{i,k'}\leq t\right)\leq \sum_{k=\lfloor n/\log n \rfloor}^n P(D_k(T_{i,k'}) \geq D_\ast,T_{i,k'}\leq t).
\end{align}
We will make yet one more union bound.
Fix $n$ and suppose $x_k$ is a height from which there are more than $D_\ast$ downward steps by time $t$.
If $D_{k-i}>D_\ast/2$ for all $i=\lfloor 8tn^2/D_\ast \rfloor,\ldots ,\lceil 16tn^2/D_\ast \rceil$, then this accounts for more than $4tn^2$ steps altogether.  
Since $\Delta t=1/(4n^2)$ it is impossible for there to be so many steps by time $t$.
It follows that
\begin{align} \label{nStarEq1}
P\left(D_k(T_{i,k'} )\geq D_\ast,T_{i,k'}\leq t\right) 
\leq \sum_{j=\lfloor 8tn^2/D_\ast \rfloor}^{\lceil 16tn^2/D_\ast\rceil} P\left(D_k(T_{i,k'}) \geq D_\ast ,D_{k-j}(T_{i,k'})\leq  \frac{D_\ast}{2}  \right).
\end{align}
If $k$ is so small so that $k-j  <n/\log n$ for one of the terms in this sum, then we can change $D_{k-j}$ to $D_{k+j}$ in the above formula.
Since we assumed in the lemma statement that $D_\ast \geq   t n (\log n)^3 $, for sufficiently large $n$ the indices of the sum must be between $n/\log n$ and $n$ for at least one of the options $k\pm j$. 
In both cases the proof is similar, so we will assume equation (\ref{nStarEq1}).

Since $D_\ast \geq   t n (\log n)^3$, the indices $k$ and $k-i$ in display (\ref{nStarEq1}) differ by at most $ n/  ( \log n)^{ 3} $, so by lemma \ref{DeltaSLemma} we have for sufficiently large $n$ that 
\begin{align}
 \Delta S(k-i,k)  \in[0.99,1.01],\qquad \text{where} \qquad  \Delta S(k,k')= \prod_{j=k}^{k'-1} \frac{p_j}{q_j} .
\end{align}
Also $k$ and $k-i$ differ by at least $ n/(2  ( \log n)^{ 3})$, so
\begin{align}
S(k-i,k)\geq \frac{n}{2C(\log n)^2}\times 0.99 \gg 100, \qquad \text{where}\qquad S(k,k')=\sum_{i=k}^{k'}\Delta S(i,k').
\end{align}
Therefore for sufficiently large $n$ the hypotheses of  lemma (\ref{DeltaNLemma}) are satisfied so we get   
\begin{align}
&P\left(D_k(T_{i,k'}) \geq D_\ast ,D_{k-i}(T_{i,k'})\leq \frac{D_\ast}{2} \right) \\
&\leq P\left(D_{k-i}(T_{i,k'})\leq \frac{D_\ast}{2} \middle| D_k(T_{i,k'}) = D_\ast   \right)
\leq  \exp\left(\frac{-0.05 D_\ast}{ S(k-i,k)}\right),\\
& \text{and }S(k-i,k)\leq 1.01 \times  16 tn^2/D_\ast.
\end{align}
This gives the lemma statement, except an extra factor of $n^4$ from union bounds.  
For sufficiently large $n$ this can be absorbed by slightly decreasing the constant in the exponent.
\end{proof}


\subsection{Approximate equality of jump times for $X$ and $X^{(n)}$} \label{ApproxJumpTimesSection}

\begin{proposition}  \label{TimeApproxProp}
For almost all $W$, all $n\in \mathbb{N}$, and for all $i\in \mathbb{N}$ we have
\begin{align}
P(|\tau_{i} -\wt{\tau}_{i}|\geq t_i n^{-0.99} , \ul{x}(i) \geq 1/\log n, \ol{D}_{i}\leq n (\log n)^3|W)
\leq \exp(-cn^{0.01})
\end{align}
where $c=c(W, \beta)$.
\end{proposition}
\begin{proof}
By lemma \ref{exitTimeLemma01} we have
\begin{align}
\E\left[\Delta \tau_i -\Delta \wt{\tau}_i\middle| x(i)=x_k ,W\right]
=&
\theta_k +O( (\log n)^3/n^3),\qquad \text{where }\theta_k=\frac{\gamma_k -\gamma_{k-1}}{4\sqrt{\beta  }x_{k}^{3/2}n^{5/2}}.
\end{align}
Suppose $x(i)=x_k\geq 1/\log n$.
Recall that $4 n^2 \Delta \tau_i$ is a geometric random variable with parameter $x_k/2$.
Using lemma \ref{exitTimeLemma01} and reasoning as in lemma \ref{mgfLemma01}, it follows that for all $k\geq n/\log n$ we have
\begin{align}
&|M(\lambda)-1-M'(0)\lambda|\leq C\log(n)^{2}n^{-4} \text{ for all }\lambda \leq \frac{Cn^2}{\log n},\\
&\text{ where }
M(\lambda)=\E[\exp(\lambda(\Delta \tau_i -\Delta \wt{\tau}_i - \theta_{k(i)}))|W,x(i)=x_k].
\end{align}
The above formula is valid as long as $x_k\geq 1/\log n$.
Now let 
\begin{align}
Z_j(\lambda)=\exp(\lambda(\tau_j -\wt{\tau}_j-\theta_{k(j)})).
\end{align}
For some universal constant $p$ we have
\begin{align}
\E\left[\one_{\ul{x}(j)\geq \log(n)^{-1}} Z_j(\lambda)\middle| W\right]
\leq &
\E\left[\one_{\ul{x}(j-1)\geq \log(n)^{-1}} Z_{j-1}(\lambda) \E[\exp(\lambda (\Delta \tau_j -\Delta \wt{\tau}_j -\theta_{k(i)})) | W,x(j)]\middle| W\right] \nonumber \\
\leq & 
\E\left[\one_{\ul{x}(j-1)\geq \log(n)^{-1}} Z_{j-1}(\lambda)\middle| W\right](1+C\log(n)^p (\lambda n^{-3}+\lambda^2 n^{-4})).
\end{align}
It follows by induction that
\begin{align}
\E\left[\one_{\ul{x}(j)\geq \log(n)^{-1}} Z_j(\lambda)\middle| W\right] \leq   \exp\left(C t (\log n)^p (\lambda n^{-1} +\lambda^2 n^{-2})\right),
\end{align}
with $C$ depending only on $W,\beta$.
Applying Chebychev's inequality as in proposition \ref{HMartProp}, and using $\lambda=n/(t\vee 1)$ we obtain
\begin{align} \label{p62}
P\bigg( \tau_i -\wt{\tau}_i -\sum_{j=1}^{i}\theta_{k(j)}\geq t n^{-0.99},\ul{x}(i)\geq 1/\log n \bigg| W\bigg) \leq \exp(-c   n^{0.01}),
\end{align}
for some constant $c=c(W,\beta )$.
By the same argument but using $1/Z_i(\lambda)$, the same bound holds for $-\tau_i +\wt{\tau}_i+\sum_{j=1}^i \theta_{k(j)}$.

To finish the proof we must show that $|\sum_{j=1}^i \theta_{k(j)}|\leq n^{-0.99}$ with high probability.
The sum can be separated as $\sum_{j=1}^i \theta_{k(j)} \one_{k(j)\leq k'} +\sum_{j=1}^i \theta_{k(j)} \one_{k(j)\geq k'}$, and we will only consider the ``$\leq$" term since the same bound can be obtained for the ``$\geq$" term similarly using the branching formula (\ref{branching2}) for numbers of upward steps.

For some $k',i'$ we have $\tau_i=T_{i',k'}$.
Let $D_k=D_k(\tau_i)$ and $d_k=d_{k}(\tau_i)$.
Summing by parts as in lemma \ref{sumByPartsLemma} we get
\begin{align}
\sum_{j=1}^i \theta_{k(j)} \one_{k(j)<k'} =&\sum_{k=1}^{k'} (D_k+D_{k+1}+b_k) \theta_k\\ =&\bigg( \sum_{k=1}^{k'} (D_{k+1}-D_k) \frac{\gamma_k +\gamma_{k-1}}{4\sqrt{\beta}x_{k}^{3/2}n^{-5/2}}\bigg)+O(D_\ast (\log n)^p /n^{ 2}).
\end{align}
As in lemma \ref{sumByPartsLemma}, the $b_k$ term became part of the $O(D_\ast (\log n)^p /n^{ 2})$; on the event $\ol{D}_i\leq n \log n$ this error term is within the tolerance of the lemma statement.

Now let
\begin{align}
Z_k (\lambda)=& \exp \bigg(\lambda \sum_{j=k}^{k'} (D_{k+1}-D_k)\frac{f_k}{n^{5/2}}\bigg),\qquad f_k =\frac{\gamma_k +\gamma_{k-1}}{4\sqrt{\beta}x_{k}^{3/2}}.
\end{align}
For $k\geq n/\log n$ we have the following estimate:
\begin{align}
\E\left[Z_k \one_{\ol{D}_k\leq D_\ast} \middle|W\right]
\leq & 
\E\left[Z_{k+1} \one_{\ol{D}_{k+1}\leq D_\ast}
\E\left[ \exp(\lambda(D_{k+1}-D_k)f_k n^{-5/2})) \middle|D_{k+1},W\right] \middle|W\right] \\
=&
\E\left[Z_{k+1} \one_{\ol{D}_{k+1}\leq D_\ast}
\left(1+\frac{q_k}{p_k}\left(\exp(-\lambda f_k n^{-5/2})-1\right) \right)^{-D_{k+1}}
 \middle|W\right] \\
\leq &
\E\left[Z_{k+1} \one_{\ol{D}_{k+1}\leq D_\ast}
 \middle|W\right] \exp\left(D_\ast C\log(n)^p (\lambda n^{-3} +\lambda^2 n^{-5})\right).
\end{align}
By induction we get (again for $k\geq n/\log n$)
\begin{align}
\E\left[Z_k \one_{\ol{D}_k\leq D_\ast} \middle|W\right]
\leq & 
\exp\left(D_\ast C\log(n)^p (\lambda n^{-2} +\lambda^2 n^{-4})\right).
\end{align}
Applying Chebychev's inequality and using $\lambda =n$ we get
\begin{align}
&\log P\bigg( \sum_{j=1}^i \theta_{k(i)}\one_{k(i)\leq k'} \geq n^{-0.99},\ol{D}_i \leq n (\log n)^3,\ul{x}(i)\geq \log(n)^{-1} \bigg|W\bigg)\\
&\leq
\log P\bigg( \sum_{j=1}^i \theta_{k(i)}\one_{\frac{n}{\log n}\leq k(i)\leq k'} \geq n^{-0.99},\ol{D}_i \leq n \log n  \bigg|W\bigg)\\
& \leq 
-\lambda n^{-0.99}+\E\left[Z_k \one_{\ol{D}_k\leq D_\ast} \middle|W\right]
\leq
-n^{0.01}+ C\log(n)^p \left( 1+n^{-1}\right) . \label{p63}
\end{align}
One can obtain the same bound for $-\sum_{j=1}^i \theta_{k(i)}\one_{k(i)\leq k'}$ using the same argument but with $1/Z_k(\lambda)$.
The lemma statement now follows from equations (\ref{p62}) and (\ref{p63}).
\end{proof}

\subsection{Combining the pieces to bound $|\Phi^{B}_n -\wt{\Phi}^{B}_n|$.} \label{PhiNPhiTildeProofSection}

\begin{proof}[Proof of theorem \ref{PhiNPhiTildeThm}]

By basic properties of sets and probability we have
\begin{align}
&P\left( \left| \Phi^{B}_n(t)-\wt{\Phi}^{B}_n(t)\right| \geq n^{-.499} , \ul{X}(1.01 t)\geq 1/\log n\middle| W\right) \\
&\leq 
 P\left( \left| \Phi^{B}_n(t)-\wt{\Phi}^{B}_n(t)\right| \geq n^{-.499} , \ul{X}(t)\geq 1/\log n,\ol{D}(t)\leq t n (\log n)^3  \middle| W\right)\\
&\qquad+P\left(\ul{X}(t)\geq 1/\log n,\ol{D}(t)\geq tn (\log n)^3   \middle| W\right) \label{a918}
\end{align}
The term on line (\ref{a918}) converges to zero by proposition \ref{NStarLemma}.
Let $i=i(t,W,B)$ be the greatest integer $i$ such that $\tau_i \leq t$.
Define terms $\RN{1},\ldots ,\RN{4}$ by
\begin{align}
\left| \Phi_n(t) -\wt{\Phi}_n(t)\right|
\leq &
\left| \Phi_n(t) -\Phi_n(\tau_i)\right|
+\left| \Phi_n(\tau_i) -\wt{\Phi}_n(\wt{\tau}_i)-\mathcal{E}_i\right|
+\left| \mathcal{E}_i \right|
+\left| \wt{\Phi}_n(\wt{\tau}_i) -\wt{\Phi}_n(t)\right|
\\
&=  \RN{1}+\RN{2}+\RN{3}+\RN{4}.
\end{align} 
By a union bound and proposition \ref{HMartProp} we have
\begin{align}
& P\left(\RN{2} \geq n^{-.499} , \ul{X}(t)\geq \log(n)^{-1},\ol{D}(t)\leq n (\log n)^3  \middle| W\right)\\
&\leq 
\sum_{j=1}^{4tn^2} P\left(\left| \Phi_n(\tau_j) -\wt{\Phi}_n(\wt{\tau}_j)-\mathcal{E}_j\right| \geq \medmath{\frac{1}{4}}n^{-.499}, \ul{X}(t)\geq \log(n)^{-1},\ol{D}(t)\leq n \log(n)   \middle| W\right) \\
& \leq n^2 \exp\left(-c n^{0.001}\right).
\label{a913}
\end{align}
By construction of $i$ and the definitions of $\Phi_n$ and $\wt{\Phi}_n$ the same estimate (\ref{a913}) holds for term $\RN{1}$, and by proposition \ref{VMartProp} we get
\begin{align}
P\left(\RN{3} \geq n^{-.249} , \ul{X}(t)\geq \log(n)^{-1},\ol{D}(t)\leq n (\log n)^3  \middle| W\right)
& \leq n^2 \exp\left(-c n^{0.001}\right).
\end{align}
By definition of $\wt{\Phi}_n$ it is clear that 
\begin{align}
\left| \wt{\Phi}_n(\wt{\tau}_i) -\wt{\Phi}_n(t)\right|
\leq C\left| \wt{\tau}_i-t\right| \sqrt{n},
\end{align}
where $C=C(W,\beta,a,t)$.
From this and proposition \ref{TimeApproxProp}, it follows that the estimate (\ref{a913}) holds yet again for term $\RN{4}$.
\end{proof}


\section{Comparison of $\wt{\Phi}^{B}_n$ and $\Phi ^{B}$} \label{section6}

In this section we prove the following
\begin{proposition} \label{WtPhiConverges}
On the event $\ul{X}(t)\geq 1/\log n$ we have
\begin{align}
\left|\wt{\Phi}^{B}_n(t) -\Phi^{B}(t)\right|\leq & C t n^{-0.249}.
\end{align}
This estimate holds for almost all $B,W$.  The constant $C$ depends only on $B,W,\beta$.
\end{proposition}
From the definitions of the functionals we have
\begin{align}
\wt{\Phi}^{B}_n(t) -\Phi^{B}(t) 
=&
\sum_{k=0}^{n-1}  \frac{B_k}{\sqrt{\beta}} +\frac{  G^{(2)}_k}{\beta x_k}   L_X([x_{k},x_{k+1}),t), \qquad \text{where}
 \label{ABdef}\\
B_k=& \frac{\sqrt{n}(G_k+G_{k+1})}{2\sqrt{x_{k}}}L_X([x_{k},x_{k+1}),t) -\int_{x_{k}}^{x_{k+1}} \frac{L_X(x,t)}{\sqrt{x}}\circ dW(x).\nonumber
\end{align}
To prove proposition \ref{WtPhiConverges}, we must estimate the right hand side of (\ref{ABdef}).

\subsection{Convergence of $\sum_k B_k$}

For the $B_k$ terms we must estimate some stochastic integrals, and to do this we will rely on the following lemma.
This kind of result is very well known, with the ideas going back at least to \cite{McKeanBook}.
\begin{lemma} \label{ItoIntEst1}
Let $W_t$ be the Wiener process and $Y^{(n)}_t$ be a sequence of stochastic processes adapted to the filtration generated by $W$.
If for each $n$ we have $|Y^{(n)}_t| \leq f(n)$ almost surely, then the bound 
\begin{align}
\left| \int_{0}^t Y^{(n)}_s \, dW_s \right|  \leq 4 f(n) \sqrt{\log n}
\end{align}
holds almost surely as well.
\end{lemma}
\begin{proof}
Consider the martingales
\begin{align}
Z^{(n)}_t =& \exp\left( \lambda \int_{0}^t Y^{(n)}_s dW_s - \frac{\lambda^2}{2}\int_{0}^t (Y^{(n)}_s)^2\, ds \right).
\end{align}
Applying Doob's martingale inequality we get that
\begin{align}
P\left(  \int_{0}^t Y^{(n)}_s dW_s - \frac{\lambda}{2}\int_{0}^t (Y^{(n)}_s)^2\, ds > r\right) \leq & e^{-\lambda r}, 
\end{align}
and it follows that
\begin{align}
P\left( \int_{0}^t Y^{(n)}_s dW_s > 2r\right) \leq e^{-\lambda r} + P\left( \int_{0}^t (Y^{(n)}_s)^2\, ds > \frac{2r}{\lambda} \right). \label{eq183561}
\end{align}
Let $r =2f(n)\sqrt{\log n}$ and $\lambda =\sqrt{\log n}/f(n)$.
The right hand side of (\ref{eq183561}) is then bounded by $1/n^2$, so our assertion follows from the Borel-Cantelli lemma.
\end{proof}

Recall that
\begin{align}
L_X(x,t) =&\frac{L_B(s(x),T^{-1}(t))}{2\exp(I(x))},\qquad I(x)= \int_{x}^1 \frac{2\, dW(y)}{\sqrt{\beta y}}.
\end{align}
The variables $G_k$ are defined (eq.\ \ref{GDEF}) in terms of $W$, but are not increments of $W$.
Therefore it is useful to decompose $G_k$ as follows
\begin{align}
G_{k}=&g_k +\wt{g}_k,\qquad g_k=n^{3/2}\int_{x_{k-1}}^{x_k}   (x-x_{k-1})\,dW(x),\qquad \wt{g}_k =n^{3/2}\int_{x_k}^{x_{k+1}}  (x_{k+1}-x)\, dW(x).
\end{align}
The point is that increments of $W$ can be expressed as
$W(x_{k})-W(x_{k-1})=n^{-1/2}(\wt{g}_{k-1}+g_k)$.
For the calculations below to fit reasonably on the page, we will need the following abbreviations:
\begin{align}
W_k=&W(x_k),\qquad \Delta W_k =W_{k+1}-W_{k},\qquad I_k =I(x_k),\\
\qquad L_B s(x)=&L_B(s(x),T^{-1}(t)),\qquad L_B s_k =L_B(s_k,T^{-1}(t)). \nonumber
\end{align}
It is straitforward to check that
\begin{align} \label{q1001}
&\sum_{k=1}^n  \frac{\sqrt{n}(G_k +G_{k-1})}{2\sqrt{x_{k-1}}  }
 L_X([x_{k-1},x_{k}),t)\\
&\qquad = \sum_{k=0}^{n-1} \bigg\{ \frac{\Delta W_{k-1} L_Bs_k}{2\sqrt{x_k}e^{ I_k}} 
+\frac{\sqrt{n}g_{k}}{4\sqrt{x_k}} \left( \int_{x_{k-1}}^{x_{k+1}} \frac{L_Bs(x)}{e^{I(x)}}\, dx -\frac{2L_B s_k}{ne^{I_k} }\right) \nonumber \\
&\hspace{100pt} +
\frac{\sqrt{n}\wt{g}_{k-1}}{4\sqrt{x_k}} \left( \int_{x_{k-2}}^{x_{k}} \frac{L_Bs(x)}{e^{ I(x)}}\, dx -\frac{2L_B s_k}{ne^{I_k}  }\right) +R_k\hspace{1pt} \bigg\},\nonumber
\end{align}
where for almost all $W$ the remainder term satisfies 
\begin{align} \label{IgnoreR}
R_k \leq C \log(n)^p n^{-3/2}\sup_{x\in \mathbb{R}}L_B(x,T^{-1}(t)),
\end{align}
where $p$ is a universal constant and $C=C(W,\beta)$.
Thus we may ignore the remainders $R_k$.

The terms $B_k$ are the left hand side of (\ref{q1001}) minus the following Stratonovich integral, which we rewrite in terms of an anti-Ito integral:
\begin{align} \label{q1002}
\int_{x_k}^{x_{k+1}} \frac{L_B s(x)}{2\sqrt{x}e^{ I(x)}} \circ dW(x)
=&
\int_{x_k}^{x_{k+1}} \frac{L_B s(x)}{2\sqrt{x}e^{ I(x)}}   d\overleftarrow{W}(x) 
- \int_{x_k}^{x_{k+1}} \frac{L_B s(x)}{2\sqrt{\beta}x e^{ I(x)}}  dx.
\end{align}
Recall that $s(x)$ and $I(x)$ are functionals of $\{ W(y):\; y\geq x\}$, so that the integrand above is nonanticipating for an anti-Ito integral but not an Ito integral.
We now show that the first term on the right hand side of (\ref{q1001}) converges to the anti-Ito integral in (\ref{q1002}):
%
\begin{lemma} \label{AntiItoConv}
On the event $\ul{X}(t)< 1/\log (n)$, the bound
\begin{align}
\left| \sum_{k=0}^{n-1} \frac{\Delta W_{k-1} L_Bs_k}{2\sqrt{x_k}e^{ I_k}}  
-
\int_{0}^{1} \frac{L_B s(x)}{2\sqrt{x}e^{ I(x)}}   d\overleftarrow{W}(x) \right| 
\leq &
C t n^{-0.49}
\label{conv2048}
\end{align}
holds for almost all $B,W$.  The constant $C$ depends only on $B,W,\beta$.
\end{lemma}
\begin{proof}
We may assume $\ul{X}(t)  \geq 1/\log n$, and thus that
\begin{align}
\sum_{k=0}^{\lfloor n/\log n\rfloor} \frac{\Delta W_{k-1} L_Bs_k}{2\sqrt{x_k}e^{ I_k}}  
=0=
\int_{0}^{1/\log n}  \frac{L_B s(x)}{2\sqrt{x}e^{ I(x)}}   d\overleftarrow{W}(x) . \label{ignoreSmallK83}
\end{align}
We may therefore assume throughout the proof that $x>1/\log n$.

Let $\lceil x\rceil_n$ indicate the smallest $x_k$ not less than $x$.  
The sum we are considering can be represented as a stochastic integral as follows:
\begin{align}
 \sum_{k=0}^{n-1} \frac{\Delta W_{k-1} L_Bs_k}{2\sqrt{x_k}e^{ I_k}}  
 =
 \int_{0}^1 \frac{L_B s(\ck{x})}{2\sqrt{\ck{x}} e^{I(\ck{x})}}\, d\ola{W}(x).
\end{align}
Thus the difference we are to estimate is 
\begin{align}
\left| \int_{0}^1 \frac{L_B s(\ck{x})}{2\sqrt{\ck{x}} e^{I(\ck{x})}} -  \frac{L_B s(x)}{2\sqrt{x}e^{ I(x)}}   d\overleftarrow{W}(x) \right| =\left| \int_{0}^1 Y^{(n)}_x \, d\ola{W}(x)\right| ,
\end{align}
and we claim that 
$
|Y^{(n)}_x|\leq Ct n^{-0.49} .
$
If we can verify this claim, then the present lemma follows from lemma \ref{ItoIntEst1}.
We emphasize that each $Y^{(n)}_x$ is nonanticipating in the $-x$ direction, so making a change of variable $x\mapsto 1-x$ transforms the expression to an ordinary Ito integral of a nonanticipating process.

By (\ref{ignoreSmallK83}) we may assume that $\sqrt{x}- \sqrt{\ck{x}}=O(\sqrt{\log n}/n)$. 
By lemma \ref{TBound} we have that the local times are up to time $T_{n}^{-1}(t)=O( (\log n)^{1.02})$, and thus by lemma \ref{LocalTimeLemma} we have $L_B s(\ck{x}) -L_B s(x) =O(n^{-0.499})$.
Lemma \ref{TBound} also gives the bound $e^{I(x)}-e^{I(\ck{x})}=O(n^{-0.499})$.
Combining these bounds finishes the proof.
\end{proof}

With lemma \ref{AntiItoConv} in hand, there remain two terms on the right hand side of (\ref{q1001}), and we now show that their sum will converge to the last integral on line (\ref{q1002}).
\begin{lemma} \label{BkLemma1111}
On the event $\ul{X}(t)\geq 1/\log n$ we have
\begin{align}
\sum_{k=0}^{n-1}\frac{\sqrt{n}g_{k}}{4\sqrt{x_k}} \left( \int_{x_{k-1}}^{x_{k+1}} \frac{L_Bs(x)}{e^{I(x)}}\, dx -\frac{2L_B s_k}{ne^{I_k}}\right)
= &
-\frac{1}{3} \int_{0}^{1} \frac{L_B s(x)}{2\sqrt{\beta}x e^{ I(x)}}  dx +O( tn^{-.249}),\label{eq857} \\
\sum_{k=0}^{n-1}\frac{\sqrt{n}\wt{g}_{k-1}}{4\sqrt{x_k}} \left( \int_{x_{k-2}}^{x_{k}} \frac{L_Bs(x)}{e^{I(x)}}\, dx -\frac{2L_B s_k}{ne^{I_k}}\right)
= &
-\frac{2}{3} \int_{0}^{1} \frac{L_B s(x)}{2\sqrt{\beta}x e^{ I(x)}}  dx+O( tn^{-.249}).
\end{align}
These approximations are valid for almost all $B,W$ and the implied constants depend on $B,W,\beta$.
\end{lemma}
\begin{proof}
The proofs for the two assertions are the same so we will only prove the first statement.
In the same way as lemma \ref{AntiItoConv}, by our assumption that $\ul{X}(t)\geq 1/\log n$, we may restrict the domain of integration to $x>1/\log n$.

Decompose terms of the sum as follows:
\begin{align}
&\frac{\sqrt{n}g_{k}}{4\sqrt{x_k}} \left( \int_{x_{k-1}}^{x_{k+1}} \frac{L_Bs(x)}{e^{I(x)}}\, dx -\frac{2L_B s_k}{ne^{I_k}}\right) \\
&\qquad =
\frac{\sqrt{n}g_{k}}{4\sqrt{x_k}}  \int_{x_{k-1}}^{x_{k+1}}e^{-I(x)} (  L_Bs(x) -L_B s_k )\,dx 
+
\frac{\sqrt{n}g_{k}}{4\sqrt{x_k}}   L_Bs_k\int_{x_{k-1}}^{x_{k+1}} \left( e^{- I(x)}  - e^{- I_k}\right)\, dx \nonumber \\
&\qquad =A'_k +B'_k.
\end{align}
First we consider the $A'$ terms.
The idea is to express the $g_k$'s as stochastic integrals and then apply lemma \ref{ItoIntEst1}.
However, the resulting integrand has a small anticipating component which we must remove and analyze separately. 
Define the following modifications of the scale function:
\begin{align}
\wt{s}(x)=-\int_{x}^1 y^{-1}\exp\left(\begin{cases} I(\ck{y}) & \text{ if }y < \ck{x} \\ I(y) & \text{ else}\end{cases} \right)\, dy
\end{align}
Note that since $I(y)$ is in the $\sigma$-field generated by $\{W_z:\; y\leq z\leq 1\}$, each $g_k$ is independent of all $\wt{s}(x)$ for $x>x_{k-1}$.
Furthermore, assuming that $x>1/\log n$ we have by lemma \ref{TBound} that 
\begin{align}
|s(x)-\wt{s}(x)|\leq C (\log n)^p n^{-3/2} 
\end{align}
for some universal constant $p$.
We decompose $A'_k$ into nonanticipating and small parts as follows:
\begin{align}
A'_k =& \frac{\sqrt{n}g_{k}}{4\sqrt{x_k}}  \int_{x_{k-1}}^{x_{k+1}}\frac{ L_B \wt{s}(x) -L_B s_k }{e^{I(\ck{x})} }\,dx + 
\frac{\sqrt{n}g_{k}}{4\sqrt{x_k}}  \int_{x_{k-1}}^{x_{k+1}}\left(\frac{ L_B s(x)   }{e^{I(x)} } -\frac{ L_B \wt{s}(x)  }{e^{I(\ck{x})} }\right)\,dx \nonumber \\
=&A'_{k,1}+A'_{k,2}.
\end{align}
Using the condition $x>1/\log n$,  lemmas \ref{TBound} and \ref{LocalTimeLemma}, and also our bound on $|s(x)-\wt{s}|$; it is easy to see that $A'_{k,2} \leq C (\log n)^p n^{-5/4}$ for some constant $C$ and universal constant $p$.
Thus the sum of $A'_{k,2}$ terms is within the tolerance of the lemma statement.

We rewrite $A'_{k,1}$ as a stochastic integral:
\begin{align}
\sum_{k=1}^n A'_{k,1} =\int_{0}^1 Y_x \, d\ola{W}(x),\qquad Y_x =
\frac{n^2 (x-\ck{x})}{4\sqrt{\ck{x}}}
\int_{\fk{x}}^{\ck{x}+1/n}  \frac{ L_B \wt{s}(y) -L_B s(\ck{x}) }{e^{I(\ck{y})} }\,dy 
\end{align}
Applying lemmas \ref{LocalTimeLemma} and \ref{ItoIntEst1}, one can see that 
\begin{align}
\sum_{k=1}^n A'_{k,1} \leq C t(\log n)^p n^{-1/2}
\end{align}
for some constant $C$ and universal constant $p$.

We now turn our attention to the $B'_k$ terms.
Substituting a linear Taylor approximation with a quadratic error term for the exponential function in $\exp (-I(x))$ it is easy to see that
\begin{align}
B'_k =&  \frac{\sqrt{n} L_Bs_k}{2\sqrt{\beta}x_k e^{I_k}} g_{k}\int_{x_{k-1}}^{x_{k+1}}\int_{x_k}^x dW(y)\, dx +R'_k,
\end{align}
where $R'_k$ satisfies the estimate (\ref{IgnoreR}) and thus can be ignored.
Observe that
\begin{align}
&\E\left[g_{k}\int_{x_{k-1}}^{x_{k+1}}\int_{x_k}^x dW(y)\, dx\right]\\
&=
\E\left[\int_{x_{k-1}}^{x_{k}} n^{3/2}(x-x_{k-1})dW(x) \left( \int_{x_{k-1}}^{x_k}(x_{k-1}-x)\, dW(x) + \int_{x_{k}}^{x_{k+1}} (x_{k+1}-x) \, dW(x) \right) \right]\\
& =
  -n^{3/2} \int_{x_{k-1}}^{x_{k}} (x-x_{k-1})^2 \, dx 
=
-\frac{1}{3n^{3/2}}.
\end{align}
In fact, a short calculation gives the following decomposition:
\begin{align}
n^{3/2} g_{k}\int_{x_{k-1}}^{x_{k+1}}  \int_{x}^{x_k}dW(y)\, dx
=& \frac{-1}{3} + n^{1/2} \int_{x_{k-1}}^{x_{k}}   Y_x \, d\ola{W}(x) \\
Y_x=&-2n^{5/2}\int_{x }^{\ck{x}} (x-\fk{x})(y-\fk{x})\, d\ola{W}(y)\\
&\qquad + n^{5/2}\int_{\ck{x}}^{\ck{x}+1/n}(x-\fk{x})(\ck{x}+1/n-y)\, d\ola{W}(y)
\end{align}
The important property of this decomposition is that for each $x$ we have $Y_x$ measurable with respect to $\{W_y:\; y\in (x,1)\}$.
Notice also that we split up the power of $n$ so that $Y_x$ is typically order 1.
We can then write the sum of $B'_k$ terms as follows
\begin{align}
\sum_{k=1}^n B'_k = \frac{-1}{3}\int_{0}^1 \frac{L_B s(\ck{x}) \, dx}{2\sqrt{\beta}\ck{x} e^{I(\ck{x})}} 
+  \frac{1}{\sqrt{n}} \int_{0}^1  \frac{ L_B s(\ck{x}) Y_x }{2\sqrt{\beta}\ck{x} e^{I(\ck{x})}}  \, d\ola{W}(x) + \sum_k R'_k \label{aoeurab}
\end{align}
It is easy to see that first term on the right hand side is within $O(t(\log n)^p n^{-1/2})$ of the right hand side of (\ref{eq857}).
Due to the explicit factor of $1/\sqrt{n}$ and lemma \ref{ItoIntEst1}, the stochastic integral term on line (\ref{aoeurab}) is small within the tolerance of the lemma statement.
As we have already said, the terms $R'_k$ are also sufficiently small.
\end{proof}

\subsection{Convergence of the $G^{(2)}$ terms}
To complete the proof of proposition \ref{WtPhiConverges} it suffices to show the following
\begin{lemma}
On the event $\ul{X}(t)\geq 1/\log n$ we have
\begin{align}
\sum_{k=1}^n \frac{ G^{(2)}_k}{  x_k}   L_X([x_{k},x_{k+1}),t) \leq C tn^{-0.499}.
\end{align}
The above estimate holds for almost all $B,W$ and the implied constant depends on $B,W,\beta $.
\end{lemma}
\begin{proof}
As observed several times before, on the assumption $\ul{X}(t)\geq \log n$ we have $L_X(x,t)=0$ for all $x<1/\log n$ and therefore it suffices to assume that $x\geq 1/\log n$ throughout this proof.

Recall the definition
\begin{align}
G^{(2)}_k =n\int_{x_{k-1}}^{x_{k+1}} \int_{x}^{x_{k+1}} f_2(y,x) \, d\ola{W}(y)\, d\ola{W}(x),
\end{align}
where $f_2$ is the function defined in display (\ref{GDEF}).
By lemma \ref{TBound} we have that $|f_2|\leq C(\log n)^p$ for some universal constant $p$.
A short calculation reveals that
\begin{align}
\sum_{k=1}^n \frac{G^{(2)}_k}{x_k} L_X([x_{k},x_{k+1}),t) 
=&
\int_{0}^1 (Y_x+\wt{Y}_x)\, d\ola{W}(x) +\sum_{k=1}^n \frac{G^{(2)}_k}{x_k} \int_{x_{k}}^{x_{k+1}} \frac{L_B s(x)-L_B\wt{s}(x)}{2e^{I(x)}}\, dx ,\label{qoiajf}
\end{align}
where
\begin{align}
Y_x =&  \frac{n}{\ck{x} } \int_{x}^{\ck{x}+1/n} f_2(y,x)\, d\ola{W}(y) \int_{\ck{x}}^{\ck{x}+1/n} \frac{L_B s(y)}{2e^{I(y)}}, \\
\wt{Y}_x =&  \frac{n}{\fk{x} } \int_{x}^{\ck{x}} f_2(y,x)\, d\ola{W}(y) \int_{\fk{x}}^{\ck{x}} \frac{L_B \wt{s}(y)}{2e^{I(y)}}.
\end{align}
On the right hand side of (\ref{qoiajf}), the sum comes from removing a small anticipating component from the stochastic integral just like we did before in lemma \ref{BkLemma1111}.
Using the bound $|s(x)-\wt{s}(x)|\leq C (\log n)^pn^{-3/2}$ and lemma \ref{LocalTimeLemma}, and also bounding $|G^{(2)}_k| \leq C(\log n)^p$ almost surely, we get that
\begin{align}
 \sum_{k=1}^n \frac{G^{(2)}_k}{x_k} \int_{x_{k}}^{x_{k+1}} \frac{L_B s(x)-L_B\wt{s}(x)}{2e^{I(x)}}\, dx =O(t(\log n)^p n^{-3/4}).
\end{align}
Using lemmas \ref{TBound}, \ref{LocTimeBound} and \ref{ItoIntEst1} we see that the stochastic integral on the right hand side of (\ref{qoiajf}) is of order $t(\log n)^p n^{-1/2}$.
\end{proof}

\section*{Appendix A}

\begin{proof}[Proof of lemma \ref{pkApproxLemma}]
By the definition of $p_k$ we have
\begin{align} \label{pk1}
\frac{p_k}{q_k}= \frac{s(x_k)-s(x_{k-1})}{s(x_{k+1})-s(x_k)} 
=
\frac{n\int_{x_{k-1}}^{x_{k}} \exp(-\int_{y}^{x_{k+1}}dV)\, dy}
{n\int_{x_{k}}^{x_{k+1}} \exp(-\int_{y}^{x_{k+1}}dV)\, dy}
\end{align}
By definition, the upper limits of integration for the $dV$ integrals would be $1$, but the present expression is equivalent by cancelling a common factor.

The stochastic integral in (\ref{pk1}) can be represented as
\begin{align}
-\int_{y}^{x_{k+1}} dV = \log \frac{k+1}{yn } -\frac{2}{\sqrt{\beta}} \int_{y}^{x_{k+1}} \frac{dW(x)}{\sqrt{x}}=\log \frac{k+1}{yn } +\frac{2}{\sqrt{\beta}} \wt{W}_{\log \frac{k+1}{yn}},
\end{align}
where $\wt{W}$ is another standard Brownian motion.
Since $ x_k\leq y \leq x_{k+1}$ we have that $\log \frac{k+1}{yn}=O(1/k)$ and so, by the law of the iterated logarithm, we have for almost all $W$ that 
\begin{align}
|\wt{W}_{\log \frac{k+1}{yn}} | \leq \sqrt{(2/k) \log \log k}.
\end{align}
Applying Taylor's theorem with the remainder in Lagrange form, the denominator of (\ref{pk1}) can be expressed as
\begin{align} \label{pk2}
&n\int_{x_k}^{x_{k+1}} \frac{k+1}{yn}\bigg\{ 1-\frac{2}{\sqrt{\beta}}\int_{y}^{x_{k+1}} \frac{dW(x)}{\sqrt{x}}+\frac{2}{\beta} \left(\int_{y}^{x_{k+1}} \frac{dW(x)}{\sqrt{x}}\right)^2 - \frac{ 4 e^{-\xi(y)}}{3 \beta^{3/2}}  \wt{W}_{\log \frac{k+1}{yn}}^3 \bigg\}\, dy, 
\end{align}
where $|\xi(y)|\leq   |\int_{y}^{x_{k+1}}\frac{dW(x)}{\sqrt{x}}|$.
Using Taylor's theorem again, line (\ref{pk2}) can be expressed as
\begin{align}
 1 -\frac{2n}{\sqrt{\beta}}\int_{x_k}^{x_{k+1}} \int_{y}^{x_{k+1}}\frac{dW(x)}{\sqrt{x}}\,dy  +\frac{1}{2k} +\frac{2n}{\beta}  \int_{y}^{x_{k+1}}\left(\int_{x_k}^{x_{k+1}} \frac{dW(x)}{\sqrt{x}}\right)^2 \, dy+O\left( \frac{ (\log \log k)^{3/2}}{k^{3/2}}\right).  \nonumber
\end{align}
We also have for almost all $W$ that for all $n,k$ the estimate
\begin{align}
\int_{y}^{x_{k+1}} \frac{dW(x)}{\sqrt{x}} = \frac{1}{\sqrt{x_k}}\int_{y}^{x_{k+1}}dW +O(\sqrt{\log k}/k^{3/2})
\end{align}
holds.
This can be checked by observing that the error term is a Gaussian random variable with variance of order $k^{-3/2}$, then applying a tail bound and Borel-Cantelli.

Making the same manipulations, the numerator of (\ref{pk1}) can be expressed as
\begin{align}
 1+\frac{3}{2k} +\int_{x_{k-1}}^y \bigg\{- \frac{2}{\sqrt{\beta x_k}}\int_{y}^{x_{k+1}} dW +\frac{2}{\beta x_k} \left(\int_{y}^{x_{k+1} }dW  \right)^2 \bigg\}\, dy
+O\left(\frac{\sqrt{\log k}}{k^{3/2}}\right).
\end{align}
Using our approximations for the numerator and denominator in (\ref{pk1}), we get
\begin{align} \label{pk3}
\log \frac{p_k}{q_k}
\sim & -\frac{2n}{\sqrt{\beta x_k}} (I_1 -I'_{1})  +\frac{1}{x_k  n}  +\frac{2n}{\beta x_k} (I_2 -I'_2) - \frac{2n^2}{\beta x_k} (I^{2}_1 -(I'_1)^2) 
+O\left(\frac{\sqrt{\log k}}{k^{3/2}}\right),
\end{align}
where
\begin{align}
I_1 =& \int_{x_{k-1}}^{x_{k+1}}(s\wedge x_k -x_{k-1})\, dW(s) ,\qquad 
I'_1  =\int_{x_{k-1}}^{x_{k+1}}  (s\vee x_k -x_k)\, dW(s)\\
I_2 =&\int_{x_{k-1}}^{x_{k+1}} \int_{x_{k-1}}^{x_{k+1}} (s_1 \wedge s_2 \wedge x_k -x_{k-1})\, dW(s_1)\, dW(s_2) , \qquad
\mathbb{E} I_2  =\frac{3}{2n^2}\\
I'_{2} =& \int_{x_k}^{x_{k+1}} \int_{x_k}^{x_{k+1}} (s_1 \wedge s_2 -x_{k})\, dW(s_1)\, dW(s_1),
\qquad \mathbb{E} I'_2 =\frac{1}{2n^2}
\end{align}
Double Ito integrals of a symmetric function can be decomposed by 
\begin{align}
\int_{0}^1 \int_{0}^1 f(s,t)\, dW(s)\, dW(t) =2 \int_{0}^1 \int_{0}^t f(s,t)\, dW(s)\, dW(t) + \int_{0}^1 f(s,s)\, ds.\label{aeorbga}
\end{align}
The point of doing this is that the stochastic integral on the right hand side of (\ref{aeorbga}) has mean zero.
Using (\ref{aeorbga}) in (\ref{pk3}) some straitforward calculations give the formulas in our lemma statement.
\end{proof}

\end{document}